\documentclass[11pt,a4paper,final]{amsart}
\usepackage[latin1]{inputenc}
\usepackage{amsmath}
\usepackage{amsfonts}
\usepackage{amssymb}
\usepackage{amsthm}
\usepackage{graphicx}
\usepackage{a4wide}
\usepackage{todo}
\usepackage{latexsym}
\usepackage{cite}
\usepackage{mathrsfs}
\usepackage{bbm}
\usepackage{subfigure}
\usepackage{float}
\usepackage{epsfig}
\usepackage{epstopdf}
\usepackage{mathtools}
\usepackage{listings}
\usepackage{pgf}
\usepackage[T1]{fontenc}
\usepackage[latin1]{inputenc}
\usepackage[english]{babel}
\usepackage{pifont,tabularx}
\usepackage{stmaryrd}
\usepackage{dsfont}
\usepackage{theoremref}
\usepackage{nicefrac}
\usepackage{enumerate}
\usepackage{natbib}
\usepackage{subfigure}

\usepackage{tikz}
\usepackage{pgfplots}
\usepackage{pgfplotstable}
\pgfplotsset{compat=1.7}
\usepackage{caption}
\usepackage{hyperref}
\hypersetup{
	colorlinks,
	citecolor=blue,
	linkcolor=blue,
	urlcolor=green}

\numberwithin{equation}{section}

\theoremstyle{plain} 
\newtheorem{thm}{Theorem}[section]
\newtheorem{prop}[thm]{Proposition}

\newtheorem{lem}[thm]{Lemma}

\newtheorem{obs}[thm]{Remark}

\theoremstyle{remark}

\usepackage{graphicx} 

\title{Collisions of random walks on comb graphs with a planar base}

\author{Umberto de Ambroggio}
\address{Department of Mathematics, National University of Singapore}
\curraddr{10 Lower Kent Ridge Road, National University of Singapore}
\email{umberto@nus.edu.sg}
\thanks{}

\author{Maximilian Nitzschner}
\address{Department of Mathematics, The Hong Kong University of Science and Technology}
\curraddr{Clear Water Bay, Kowloon, Hong Kong}
\email{mnitzschner@ust.hk}
\thanks{}

\author{Carlo Scali}
\address{School of Computation, Information and Technology, Technische Universit{\"a}t M{\"u}nchen}
\curraddr{Boltzmannstra{\ss}e 3, 85748, Garching bei M{\"u}nchen}
\email{carlo.scali@tum.de}
\thanks{}

\date{\today}

\begin{document}
\begin{abstract}
In this article we study collisions of two independent random walks on comb graphs $\mathrm{Comb}(\tilde{G},f)$ for a large class of recurrent planar graphs $\tilde{G}$ and profile functions $f$, the latter governing the length of vertical segments (called ``teeth'') attached to vertices of the base graph $\tilde{G}$. We prove that the number of collisions of two random walks starting from the same site undergoes a phase transition depending on the growth of $f$. As a benchmark example, we show that for $\mathrm{Comb}(\mathbb{Z}^2,f_\gamma)$ with $f_\gamma(z) = \log^{\gamma}(\|z\|_\infty\vee 1)$ and $\|\cdot \|_\infty$ denoting the supremum norm, two independent random walks started at the origin collide \textit{finitely} often almost surely if $\gamma > 1$, answering a question of Barlow, Peres, and Sousi~\cite{BPS}, who established that infinitely many collisions occur almost surely if $\gamma \leq 1$. We furthermore establish phase transitions in the cases where the base graph $\tilde{G}$ is pre-fractal, or a typical realization of a supercritical cluster of planar Bernoulli bond percolation. \smallskip


\end{abstract}
\maketitle

\section{Introduction}
\label{sec:Intro}
In the present article we investigate collisions of independent (discrete-time) random walks on a large class of comb-like graphs where the base graph is planar. Let $G = (V,E)$ be a connected, locally finite, infinite graph with vertex set $V$ and edge set $E$, such that a single random walk on $G$ is recurrent. We say that $G$ has the \textit{infinite collision property} if two independent simple random walks on $G$ starting from the same vertex collide infinitely many times, almost surely. \medskip

While for \textit{transitive} recurrent graphs $G$ the infinite collision property always holds, the article~\cite{KP} provides a surprising generic example of a class of recurrent graphs $G$ in which two independent simple random walks collide only \textit{finitely many times}, almost surely. These graphs $G$ considered in~\cite{KP} are comb-like graphs $\mathrm{Comb}(\tilde{G},\mathbb{Z})$, which are constructed from a given regular  recurrent graph $\tilde{G} = (\tilde{V},\tilde{E})$ by attaching copies of $\mathbb{Z}$ to each vertex. As a refinement, one can consider for a given \textit{profile function} $f : \tilde{V} \to \mathbb{N} \cup \{+\infty\}$ the subgraph $\mathrm{Comb}(\tilde{G},f)$ in which only vertices $(x,n)$ with $0 \leq n \leq f(x)$ of the vertex set of $\mathrm{Comb}(\tilde{G},\mathbb{Z})$ are retained. Informally, $\mathrm{Comb}(\tilde{G},f)$ can be viewed as a truncated version of $\mathrm{Comb}(\tilde{G},\mathbb{Z})$ with vertical segments (also called \textit{teeth}) of length $f(x)$ attached to each $x \in \tilde{V}$, and the question whether the infinite collision property holds for $\mathrm{Comb}(\tilde{G},f)$ depends in a delicate way on the geometry of $\tilde{G}$ and the profile function $f$. \medskip 

Collisions of two (or more) random walks have been the object of significant interest in recent years for graphs of comb-type in the case $\tilde{G} = \mathbb{Z}$, see~\cite{BPS,Chenchen,CWZ08,CroydonDeAmbroggio}, as well as for various classes of other graphs, e.g.~in~\cite{BPS,ChenChen10,HH,HP,Wat23} (a more detailed overview of the relevant literature is given later in this introduction). \medskip 

The principal contribution of the present article is to determine, for a large class of planar graphs $\tilde{G}$, the leading-order growth rate of certain profile functions $f$, depending only on the distance to a fixed vertex $o \in \tilde{V}$, that separates a phase of infinitely many collisions of two independent random walks from a phase of finitely many collisions. As a specific example, we consider $\tilde{G} = \mathbb{Z}^2$, and show that for a growth profile $ f_\gamma(z) = \lfloor\log^\gamma(\|z\|_\infty \vee 1) \rfloor$, where $\gamma \in (0,\infty)$ (with $\|\cdot\|_\infty$ denoting the supremum norm, $\lfloor \cdot \rfloor$ the lower integer part, and $a \vee b$ the maximum of two real numbers $a,b$), two random walks on the graph $\mathrm{Comb}(\mathbb{Z}^2,f_\gamma)$ starting from the same vertex collide finitely often almost surely if, and only if, $\gamma > 1$. This characterizes the phase transition up to leading order and answers a question raised by Barlow, Peres, and Sousi in~\cite[Section 6, 5.]{BPS}, where it is proved that there are infinitely many collisions almost surely when $\gamma \le 1$. We also treat a general class of comb graphs over $\tilde{G}$ where $\tilde{G}$ is a pre-fractal graph, subject to certain regularity properties, as well as the case in which $\tilde{G}$ is a (typical) realization of a  supercritical Bernoulli bond-percolation cluster in $\mathbb{Z}^2$, conditioned to contain the origin.  \medskip

We now introduce the required notation and our results in a more precise way. Given a connected, locally finite, infinite, recurrent \textit{base graph} $\tilde{G} = (\tilde{V},\tilde{E})$ and a profile function $f : \tilde{V} \to \mathbb{N} \cup \{+\infty\}$, we define the comb graph $G = \mathrm{Comb}(\tilde{G},f) = (V,E)$ to have vertex set 
\begin{equation}
    V \coloneqq \{ (v,n) \in \tilde{V} \times  \mathbb{N} \, : \, 0 \leq n \leq f(v) \},
\end{equation}
and edge set
\begin{equation}
\begin{split}
    E & \coloneqq \Big\{\{(u,0),(v,0) \} \, : \,  \{u,v\} \in \tilde{E} \Big\}  \cup \Big\{\{(v,m),(v,n) \} \, : \, v \in \tilde{V}, |m-n| = 1 \Big\}.
    \end{split}
\end{equation}
We let $(X_n)_{n\in \mathbb{N}}$ and $(Y_n)_{n \in \mathbb{N}}$ be two independent discrete-time simple random walks on $G$, starting under the probability measure $\mathbb{P}_{x,y} \coloneqq \mathbb{P}_x \otimes \mathbb{P}_y$ from $x \in V$ and $y \in V$, respectively (see Section~\ref{sec:Notation} for further details concerning the notation). 

We are interested in the total number of collisions of $X$ and $Y$ on $\mathrm{Comb}(\tilde{G},f)$, which is formally given by the $\mathbb{N} \cup \{+\infty\}$-valued random variable
\begin{equation}
    \mathcal{Z} = \sum_{n = 0}^\infty \mathds{1}{\{n \in \mathbb{N} \, : \,  X_n = Y_n \}}.
\end{equation}
We say that the graph $G$ has the \textit{infinite collision property} if 
\begin{equation}
\label{eq:Infinite-coll}
    \mathbb{P}_{x,x}(\mathcal{Z} = \infty) = 1, \qquad \text{for all }x \in V. 
\end{equation}
On the other hand, $G$ has the \textit{finite collision property} if 
\begin{equation}
\label{eq:Finite-coll}
    \mathbb{P}_{x,x}(\mathcal{Z} < \infty) = 1, \qquad \text{for all }x \in V. 
\end{equation}
By~\cite[Proposition 2.1]{BPS}, one has a $0$-$1$ law stating that
\begin{equation}\text{
        $G$ fulfills either~\eqref{eq:Infinite-coll} or~\eqref{eq:Finite-coll}}
\end{equation}
(in fact, this holds for \textit{any} locally finite, connected recurrent graph, although we will only consider graphs of comb-type in this article). Our interest is in comb graphs with a radially growing ``tooth profile'' $f$, which means that for a fixed origin $o \in \tilde{V}$, one has
\begin{equation}
 f(v) = \varrho(d(o,v)),   
\end{equation}
where $\varrho : \mathbb{N} \to \mathbb{N} \cup \{+\infty\}$ is a non-decreasing function and $(d(u,v))_{u,v \in \tilde{V}}$ is a metric on $\tilde{V}$. Typically, $d(\cdot,\cdot)$ will be chosen either as the graph distance $d_{\tilde{G}}$ on $\tilde{G}$ or the metric induced by the sup-norm distance on $\mathbb{Z}^2$. For the examples below, we usually consider for varying $\gamma \in (0,\infty)$ the functions
\begin{equation}
\label{eq:Generic-teeth}
    \begin{cases}
        \varrho_{p,\gamma}(k) = \lfloor k^\gamma \rfloor & \text{(polynomial growth), or} \\
         \varrho_{\ell,\gamma}(k) = \lfloor \log^\gamma(k \vee 1 ) \rfloor & \text{(logarithmic growth)}.
    \end{cases}
\end{equation}
Our first result concerns the comb graph with a pre-fractal recurrent base graph $\tilde{G}$, see Section~\ref{sec:Fractal-base} for a detailed description of the set-up. Essentially, we require the base graph $\tilde{G}$ to fulfill certain regularity properties~\eqref{eqn:VolGrowth} and \eqref{eqn:ResGrowth}, 
which correspond respectively to uniform, up-to-constants upper and lower bounds for the volume of graph-distance balls of radius $k \geq 0$ by $k^\alpha$ (with $\alpha \in [1,2]$), and to uniform, up-to-constants upper and lower bounds for the effective resistance between vertices $u,v \in \tilde{V}$ in terms of $d_{\tilde{G}}(u,v)^{\beta - \alpha}$ (with $\beta \in [2,\alpha+1]$, $\beta > \alpha$).
Notable examples of graphs covered by our assumptions are planar graphs associated with regular fractals such as the Sierpi{\'n}ski gasket and carpet and Vicsek sets. Let us remark that, in many of those instances, the parameters $\alpha$ and $\beta$ that characterize~\eqref{eqn:VolGrowth} and \eqref{eqn:ResGrowth} are well-understood. For example, in the case of the Sierpi{\'n}ski gasket, we have $\alpha = \log(3)/\log(2)$ and $\beta = \log(5)/\log(2)$. In Theorem~\ref{theo:Weak}, we prove that
\begin{equation}
\label{eq:Main-result-Fractal-Intro}
    \begin{minipage}{0.8\textwidth}
        Under assumptions~\eqref{eqn:VolGrowth} and \eqref{eqn:ResGrowth}, 
        $\mathrm{Comb}(\tilde{G},f_\gamma)$ with $f_\gamma(v) = \lfloor d_{\tilde{G}}(o,v)^\gamma \rfloor$ has the infinite collision property if $\gamma \leq \beta - \alpha$, and the finite collision property if $\gamma > \beta-\alpha$.
    \end{minipage}
\end{equation}
Next, we consider the comb graph with base graph $\tilde{G} = \mathbb{Z}^2$. We prove in Theorem~\ref{theo:2d} that 
\begin{equation}
\label{eq:Main-result-2D-Intro}
    \begin{minipage}{0.8\textwidth}
        $\mathrm{Comb}(\mathbb{Z}^2,f_\gamma)$ with $f_\gamma(z) = \lfloor \log^\gamma(\|z\|_\infty \vee 1) \rfloor$ has the infinite collision property if $\gamma \leq 1$, and the finite collision property if $\gamma > 1$.
    \end{minipage}
\end{equation}
While the infinite collision property for $\gamma \leq 1$ follows in a straightforward manner from the Green kernel criterion of~\cite{BPS} as was proved in the latter, the finite collision property for $\gamma > 1$ requires a fine understanding of the heat kernel on $\mathrm{Comb}(\mathbb{Z}^2,f_\gamma)$.
The logarithmic growth rate~\eqref{eq:Main-result-2D-Intro} at which the phase transition occurs should be compared with the comb graph  $\mathrm{Comb}(\mathbb{Z},f_\gamma)$, with the polynomial profile $f_\gamma(z) = \lfloor |z|^\gamma \rfloor$, where the infinite collision property holds, if and only if, $\gamma \leq 1$, see~\cite[Section 4]{BPS}. \medskip

Finally, in Section~\ref{sec:Random-base-graphs} we consider the case in which the underlying graph $\tilde{G}$ is a typical realization of a planar Bernoulli (bond) percolation cluster $\mathcal{C}_\infty$, with parameter $p \in (p_c(2),1]$, where $ p_c(2) = \frac{1}{2}$ is the percolation threshold, see the beginning of Section~\ref{sec:Random-base-graphs} for precise definitions. We show in Theorem~\ref{theo:perc} that, conditionally on $\{0 \in \mathcal{C}_\infty\}$, for almost every realization of $\mathcal{C}_\infty$ one has that
\begin{equation}
\label{eq:Main-result-perc-Intro}
    \begin{minipage}{0.8\textwidth}
        $\mathrm{Comb}(\mathcal{C}_\infty,f_\gamma)$ with $f_\gamma(z) = \lfloor \log^\gamma(\|z\|_\infty \vee 1) \rfloor$ has the infinite collision property if $\gamma \leq 1$, and the finite collision property if $\gamma > 1$.
    \end{minipage}
\end{equation}
Note that this last result includes the result for $\mathbb{Z}^2$ if we take $p = 1$. We first show the proof for the more regular case of $\mathbb{Z}^2$ and then extend it to $\mathcal{C}_\infty$, because the latter case is more technical and requires heavier notation. \medskip

\textbf{Related work.} We briefly give an account of some previous results regarding collisions of random walks, making no attempt to provide an exhaustive list. One of the first works on this topic is by Krishnapur and Peres~\cite{KP}. There, the authors prove that the graph $\mathrm{Comb}(\mathbb{Z},\mathbb{Z})$, i.e.~$\mathbb{Z}$ with (bi-)infinite comb teeth which are also copies of $\mathbb{Z}$ has the finite collision property. As pointed out in the same reference, their method applies to general comb graphs $\mathrm{Comb}(\tilde{G},\mathbb{Z})$ with (bi-)infinite teeth, provided that $\tilde{G}$ is a connected, recurrent, infinite, and regular graph. Later Barlow, Peres, and Sousi \cite{BPS} studied the transition between the finite and infinite collision property on several graphs, including $\mathrm{Comb}(\mathbb{Z},f)$ with a polynomially growing tooth profile $f$, spherically symmetric trees, and typical realizations of certain random recurrent graphs, thereby substantially strengthening a previous result by Chen, Wei, and Zhang~\cite{CWZ08} obtained for $\mathrm{Comb}(\mathbb{Z},f)$. For $\mathrm{Comb}(\mathbb{Z},f)$, Chen and Chen have studied in \cite{Chenchen} the nature of the phase transition beyond the leading order, by introducing multiplicative corrections of logarithmic order to the polynomial growth rate of $f$. We will describe part of the work~\cite{BPS} in more depth below as it is the most relevant in comparison to the present one. In \cite{HP}, Hutchcroft and Peres established a general criterion applicable to certain ``reversible'' random recurrent graphs to derive the infinite collision property. There are also several papers that explore the infinite collision property for random walks in random environment. Among others, let us cite the works of Halberstam and Hutchcroft~\cite{HH} on random dynamical conductances (in dimension 2), of Devulder, Gantert, and P\`ene~\cite{DGP1, DGP2} for the classical random walk in i.i.d.~random environment on $\mathbb{Z}$, and of Devulder~\cite{devulder2025infinitely} concerning collisions between a simple random walk and multiple random walks in a random environment on $\mathbb{Z}$. Recently, the first author and Croydon \cite{CroydonDeAmbroggio} established a similar phase transition for infinitely (respectively, finitely) many triple collisions on $\mathrm{Comb}(\mathbb{Z},f)$ with logarithmically growing tooth profiles $f$. \medskip

\textbf{Strategy of the proof.}  To put our results into context and to outline our proof strategy, we briefly return to the work \cite{BPS}. There, the authors devised a criterion that implies the infinite collision property in terms of the killed Green kernel (the precise statement is given in~\eqref{eq:Infinite-collision}). Using the latter,
they established that the comb graph $\mathrm{Comb}(\mathbb{Z},f_\gamma)$ with $f_\gamma(x) = \lfloor |x|^\gamma\rfloor, x \in \mathbb{Z}$, has the infinite collision property if $\gamma \le 1$. The Green kernel criterion of \cite{BPS} can be used in a rather straightforward manner to prove the infinite collision property, up to a certain growth rate, for many planar graphs with dimensions $1< d \le 2$, suggesting that the phase transition might occur (at leading order) when the Green kernel criterion becomes non-effective. Our results essentially confirm this picture by establishing that the Green kernel criterion of \cite{BPS} indeed governs the phase transition at the first order. \medskip

The approach in~\cite{BPS} to show the finite collision property of $\mathrm{Comb}(\mathbb{Z},f_\gamma)$ when $\gamma > 1 $ is to partition the graph into certain subgraphs and to prove that the probabilities to observe at least one collision in each such subgraph are summable. This relies on fairly precise estimates for the heat kernel on the comb graph, and most of our work focuses on deriving such estimates in the context of more complex underlying graphs $\tilde{G}$. \medskip

 In Section~\ref{sec:Fractal-base} we work on comb graphs with a base graph $\tilde{G}$ fulfilling~\eqref{eqn:VolGrowth} and \eqref{eqn:ResGrowth}, and the general analysis proceeds similarly to that for $\mathrm{Comb}(\mathbb{Z},f_\gamma)$ in~\cite{BPS}. In particular, the results from Lemma~\ref{lemma:ExitCombLDP} to Lemma~\ref{lemma:SmallTimeFract} may be viewed as analogues of similar controls established in~\cite{BPS}, in the specific context of $\tilde{G} = \mathbb{Z}$. However, somewhat surprisingly, it turns out that the ostensibly most natural partition of the base graph (that is, by concentric spheres of the base graph with their attached teeth) is \textit{not} sufficient for our purposes. This is essentially due to the lack of control on the size of spheres in fractal graphs. Instead, we re-construct balls using a ``smoothed'' version of spheres. In this procedure, while we lose a bit of control on the heat kernel, we gain regularity of the objects we are summing over, and the bounds so obtained are precise enough to conclude the proof. \medskip
 
 In Section~\ref{sec:2D-Comb} we study the comb graph with base graph $\mathbb{Z}^2$. Obtaining the appropriate heat-kernel bound on the former is the most involved part, and constitutes the technical core of the present paper. Most of this work is encapsulated in the proof of Lemma~\ref{lemma:HeatKernelZ2k}. The key insight is that the size of the teeth in a ball of radius $k \geq 1$ is at most a power of $\log(k)$ and thus much smaller than the relevant (diffusive) time scale of the walk. We are able to exploit this fact to show good concentration properties of the sum of the times spent in vertical excursions, i.e.~between consecutive horizontal steps. These concentration estimates lead to meaningful comparisons between the heat kernel and a truncated Green kernel. \medskip
 
Finally, in Section~\ref{sec:Random-base-graphs} we deal with comb graphs for which the underlying graph is the supercritical Bernoulli percolation cluster. Remarkably, the logarithmic scale at which the transition occurs (see \eqref{eq:Main-result-perc-Intro}) essentially coincides with the typical size of subsets of bad isoperimetry in a large box (see, e.g.,~\cite{abe2015effective}). It is thus not clear a priori whether such fluctuations impact the principal order of the phase transition governing the behavior on the comb over $\mathbb{Z}^2$ when moving to $\mathcal{C}_\infty$. However, we show that this is not the case. To that end, we use the classical controls in~\cite{barlow2004random} to show that heat kernels on $\mathcal{C}_\infty$ are eventually well-behaved over sizable boxes, which allows the implementation of the strategy employed for $\mathrm{Comb}(\mathbb{Z}^2,f_\gamma)$. \medskip

\textbf{Organization and conventions.} In Section~\ref{sec:Notation}, we collect further notation and background material on random walks, effective resistances, and heat kernels. In Section~\ref{sec:Fractal-base}, we prove in Theorem~\ref{theo:Weak} the phase transition~\eqref{eq:Main-result-Fractal-Intro} in terms of a polynomial growth rate for a class of graphs with a base graph fulfilling~\eqref{eqn:VolGrowth} and \eqref{eqn:ResGrowth}, including in particular many pre-fractal graphs. Section~\ref{sec:2D-Comb} deals with the benchmark case of comb graphs with base $\mathbb{Z}^2$, and our main result on the phase transition~\eqref{eq:Main-result-2D-Intro} appears in Theorem~\ref{theo:2d}. In Section~\ref{sec:Random-base-graphs} we turn to combs with a random base graph, and prove in Theorem~\ref{theo:perc} that the phase transition at a logarithmic scale is retained if one considers comb graphs in which the base graph is given by a typical realization of a supercritical cluster of Bernoulli bond percolation in $\mathbb{Z}^2$ containing the origin, thus proving~\eqref{eq:Main-result-perc-Intro}. In the Appendix~\ref{appendix} we provide proofs for certain technical random walk estimates that are used at various stages within the article. \medskip

Throughout the article, we denote generic positive constants by $C, c', \dots$, which may change their value from place to place. Dependence on further parameters will be highlighted in the notation. Numbered constants $c_1, c_2, ...$ retain the value given to them at their first appearance in the text. In order to simplify the notation, in several places throughout the paper we avoid using the (more precise) $\lfloor \cdot \rfloor,\lceil \cdot \rceil$ in statements involving lower or upper integer parts, when the meaning is clear from the context.

\section{Notation and useful results on random walks}
\label{sec:Notation}

In this section, we introduce some further notation and present various standard facts about random walks on graphs, Green's functions, and effective resistances. We also state the Green kernel criterion of~\cite{BPS} which is used to prove the infinite collision property of two independent random walks on certain comb graphs. \medskip

\textbf{Graph-theoretic terminology.} We start with some elementary notation. We write $\mathbb{N} = \{0,1,...\}$ for the set of natural numbers including zero. We denote, for numbers $a,b \in \mathbb{R}$, the maximum of $a$ and $b$ by $a \vee b$, the minimum of $a$ and $b$ by $a \wedge b$, and write $\lfloor a \rfloor$ (resp.~$\lceil a \rceil$) for the lower (resp.~upper) integer part of $a$, namely the largest integer smaller or equal to $a$ (resp.~the smallest integer larger or equal to $a$). For a connected and locally finite graph $G = (V,E)$ with vertex set $V$ and edge set $E$, we say that $u \in V$ and $v \in V$ are neighbors (denoted $u \sim v$) if, and only if, $\{u,v\} \in E$. A sequence of distinct neighbors $u_1 \sim u_2 \sim ... \sim u_n$ with $n \geq 2$ is called path of length $n-1$ between $u_1$ and $u_n$. We denote the graph distance between $u \in V$ and $v \in V$ by $d_{G}(u,v)$; this is the length of a shortest path between $u$ and $v$ (with the convention $d_{G}(u,u) = 0$). The cardinality of a set $K \subseteq V$ will be denoted by $|K|$, and if $|K|< \infty$, we write $K \subset\subset V$. For a vertex $v \in V$, we write $\mathrm{deg}(v) \coloneqq |\{u \in V \, \colon \, u  \sim v \}|$ for its degree, i.e.~the number of neighbors of $v$. For a given $x \in V$ and $r \geq 0$, we denote the (closed) ball in graph-distance as $B_G(x,r) \coloneqq \{y \in V \,  \colon \, d_G(x,y) \leq r \}$. For a set $A \subseteq V$ we write $\partial A \coloneqq \{x \in  A \, \colon \, \exists y \in V \setminus A \text{ with } y \sim x\}$ for the boundary of $A$. For two sets $A, B \subseteq V$ we write $\partial_A B$ to denote $\partial B \cap A$. When working on the integer lattice $\mathbb{Z}^2$ or subsets thereof, we slightly abuse notation and denote for instance the graph $G(\mathbb{Z}^2) = (\mathbb{Z}^2, E(\mathbb{Z}^2))$ (with the usual nearest-neighbor edge set $E(\mathbb{Z}^2)$) simply by $\mathbb{Z}^2$. We denote by $\|z\|_\infty \coloneqq |z_1| \vee |z_2|$ and $\|z\|_1 \coloneqq |z_1| + |z_2|$ the supremum norm and the $\ell^1$-norm of a point $z = (z_1,z_2) \in \mathbb{Z}^2$, respectively. When $z \in \mathbb{Z}^2$ and $r \geq 0$ we write $B_\infty(z,r) \coloneqq (z + [-r,r]^2)\cap \mathbb{Z}^2$ for the sup-norm ball of radius $\lfloor r \rfloor$ centered at $z$.  \medskip

\textbf{Random walk on graphs.} We now move to the simple random walk on graphs and related notions. Let $G = (V,E)$ be a locally finite and connected graph. For $x \in V$, we consider the probability measure $\mathbb{P}_x$ on the canonical space $V^{\mathbb{N}}$ under which the coordinate projections $(X_n)_{n \in \mathbb{N}}$ have the law of a (discrete-time) simple random walk starting from $x \in V$, that is, the discrete-time Markov chain with starting distribution $\delta_x$ (the point measure centered in $x \in V$) and transition probabilities given by
\[p_{x,y} = \frac{1}{\mathrm{deg}(x)}\mathds{1}\{y \sim x \}, \qquad x,y  \in V.\]
The expectation under $\mathbb{P}_x$ will be denoted by $\mathbb{E}_x$. We sometimes write $\mathbb{P}^G_x$ instead of $\mathbb{P}_x$ if we wish to emphasize the graph $G$ on which the random walk is defined, and a similar comment applies to $\mathbb{E}_x^G$. When considering two independent random walks, we usually work with the product measure $\mathbb{P}_x \otimes \mathbb{P}_y$ on the space $(V \times V)^{\mathbb{N}}$, and denote the canonical process of the coordinate projections by $(X_n,Y_n)_{n \in \mathbb{N}}$. We denote the family of canonical shift operators on the space of trajectories $V^{\mathbb{N}}$ by $(\theta_n)_{n \geq 0}$, where $\theta_n(w)(k) = w(n+k)$ for every $n,k \in \mathbb{N}$ and $w \in V^{\mathbb{N}}$. For any $A \subseteq V$, we write $T_A \coloneqq \inf\{n \geq 0 \, : \, X_n \notin A \}$ for the exit time of the walk from $A$ and $H_A \coloneqq \inf\{n \geq 0 \, : \, X_n \in A \}$ for the entrance time of the walk into $A$ (with the convention that $\inf \varnothing = +\infty$). When $A = \{x\}$ is a singleton set, we abbreviate $H_{\{x\}}$ by $H_x$. We furthermore define the heat kernel of the random walk $(X_n)_{n \geq 0}$ on $G = (V,E)$ by
\begin{equation*}
    p_n(x,y) \coloneqq \frac{1}{\mathrm{deg}(y)} \mathbb{P}_x(X_n = y), \qquad x,y \in V, n \in \mathbb{N},
\end{equation*}
and the Green kernel of the random walk killed upon exiting $A$ (with $A \subseteq V$) by
\begin{equation*}
    g_A(x,y) \coloneqq \frac{1}{\mathrm{deg}(y)} \sum_{n = 0}^\infty \mathbb{P}_x(X_n = y, n < T_A), \qquad  x,y \in V.
\end{equation*}
 Both $p_n$ and $g_A$ are symmetric in their arguments. We will also need several versions of a truncated Green kernel for the random walk. For integer $0 \le a \le b$, we define 
    \begin{equation*}
        g_{[a, b]}(x, y) \coloneqq  \frac{1}{\mathrm{deg}(y)}\mathbb{E}_{x}\left[ \sum_{n = a}^b \mathds{1}\{X_n = y\} \right] \left( = \sum_{n = a}^b p_n(x,y) \right), \qquad x,y \in V.
    \end{equation*}
When $b \in \mathbb{N}$ and $x,y \in V$ we simply write $g_b(x,y)$ instead of $g_{[0,b]}(x,y)$. By the symmetry of $p_n$ in its arguments, $g_{[a,b]}$ is again symmetric for any integer $0 \leq a \leq b$. We sometimes write $p^G_n(\cdot,\cdot)$, $g_A^G(\cdot,\cdot)$, or $g^{G}_{[a,b]}(\cdot,\cdot)$ for the quantities introduced above if we wish to emphasize the underlying graph on which we are operating. \medskip

We now introduce the fundamental notion of effective resistances on graphs. For a more detailed treatment, see, e.g.,~\cite{Barbook,LP}. Given a connected, locally finite graph $G = (V,E)$ and $f,g : V \to \mathbb{R}$ we define the Dirichlet form as 
\begin{equation*}
    \mathcal{E}^G(f,g) \coloneqq \frac{1}{2}\sum_{x,y \in V, x \sim y} \big(f(x) - f(y) \big) \cdot \big(g(x) - g(y) \big),
\end{equation*}
whenever the expression on the right-hand side is absolutely summable. For disjoint, non-empty subsets $A,B \subseteq V$ we then define the effective resistance between $A$ and $B$ as 
\begin{equation}
\label{eq:EffRes}
R_{\mathrm{eff}}^G(A,B) \coloneqq \inf\left\{ \mathcal{E}^G(f,f) \, \colon \, f\vert_A = 1, f\vert_B = 0, \mathcal{E}(f,f) < \infty \right\}^{-1}
\end{equation}
(recall the convention that $\inf \varnothing = +\infty$). We also abbreviate $R^G_{\mathrm{eff}}(\{x\},B)$ by $R^G_{\mathrm{eff}}(x,B)$ for $x \in V \setminus B$ and $\varnothing \neq B \subseteq V$, and similarly define $R^G_{\mathrm{eff}}(x,y) = R^G_{\mathrm{eff}}(\{x\},\{y\}) $, for $x \neq y\in V$. It follows immediately from the definition~\eqref{eq:EffRes} that $R^G_{\mathrm{eff}}(A_1,B) \geq R^G_{\mathrm{eff}}(A_2,B)$ for any $\varnothing \neq A_1 \subseteq A_2 \subseteq V$, and that $R^G_{\mathrm{eff}}(\cdot,\cdot)$ is symmetric with respect to the interchange of both arguments. For $x \in B \subset\subset V$, one has the identity
\begin{equation}
\label{eq:Green-Eff-Res}
    R^G_{\mathrm{eff}}(x,B^c) = g_B(x,x),
\end{equation}
see, e.g.,~\cite[Theorem 2.8]{Barbook}. The usefulness of~\eqref{eq:Green-Eff-Res} comes from the fact that the effective resistance is amenable to operations of network reduction (see, e.g.,~\cite[Section 2.5]{Barbook} or~\cite[Section 2.3]{LP}). A convenient dual representation for the effective resistance between disjoint, non-empty subsets $A,B \subseteq V$ with $V \setminus (A \cup B)$ finite is Thompson's principle, which states that
\begin{equation}
\label{eq:Thompson}
    R_{\mathrm{eff}}^G(A,B) = \inf\{E(I,I) \, : \, I \in \mathcal{I}_0(A,B) \},
\end{equation}
where $E(I,I) = \frac{1}{2}\sum_{x \in V}\sum_{y \sim x} I_{x,y}^2 $, and  $\mathcal{I}_0(A,B)$ is the space of \textit{flows} from  $A$ to $B$, i.e.~antisymmetric functions $I : V \times V \to \mathbb{R}$ with $I_{x,y} = 0$ if $x\not\sim y$, fulfilling $\sum_{x \in A}\sum_{y \sim x}I_{x,y} = 1$, and $\sum_{y \sim x}I_{x,y} = 0$ for $x \notin A \cup B$, see, e.g.,~\cite[Theorem 2.31]{Barbook}. We are now in the position to state the Green kernel criterion from~\cite[Theorem 3.1]{BPS}:
\begin{equation}
\label{eq:Infinite-collision}
    \begin{minipage}{0.8\linewidth}
        Suppose $G$ is recurrent,  $o \in V$ is a fixed vertex and $(D_r)_{r \in \mathbb{N}}$ is an increasing sequence of subsets $D_r \subsetneq V$ with $\bigcup_{r \in \mathbb{N}} D_r = V$. Suppose that 
        $$
        g_{D_r}(x,x) \leq Cg_{D_r}(o,o) \qquad \text{for all }x\in D_r, r \in \mathbb{N}.
        $$
        Then $G$ has the infinite collision property.
    \end{minipage}
\end{equation}
We will often appeal to~\eqref{eq:Infinite-collision} to prove the infinite collision property for a given comb graph. \medskip

For later applications, we collect several estimates for random walks on both general graphs and some further estimates for random walks on the planar integer lattice $\mathbb{Z}^2$. We will need the following general bound from~\cite[Lemma 4.4]{BPS}: For any locally finite, connected graph $G$ and $B \subseteq V$, one has
\begin{equation}
    \label{eq:Bound-HK-via-Green}
    p_n(x,x) \leq \frac{2g_B(x,x)}{n\mathbb{P}_x(T_B \geq n)}, \qquad x\in B, n \in \mathbb{N}.
\end{equation}
Furthermore, we will use that 
\begin{equation}
\label{eq:Heat-kernel-decreasing}
\begin{split}
    p_{2n+1}(x,x) & \leq p_{2n}(x,x), \\
    p_{2n+2}(x,x) & \leq p_{2n}(x,x), \qquad x \in V, n \in \mathbb{N},
    \end{split}
\end{equation}
see, e.g.,~\cite[Lemma 4.1]{Barbook}. We also note that
\begin{equation}
\label{eq:Green-function-maximal}
g_n(x,x) = \sup_{y \in V} g_n(x,y), \qquad x \in V, n \in \mathbb{N},
\end{equation}
which follows by writing $g_n(y,x) = \frac{1}{\mathrm{deg}(x)} \mathbb{E}_y\left[\sum_{k = H_{x}}^n \mathds{1}\{X_k = x \} \right]$, and using the strong Markov property, as well as the symmetry of $g_n$. \medskip

We conclude this section with a useful occupation time estimate for the planar random walk which will be used in Section~\ref{sec:2D-Comb}. \begin{lem}\label{lemma:MomentsOccTime}
    Consider the simple random walk $X$ on $\mathbb{Z}^2$. We have, for any $z \in \mathbb{Z}^2$ and $t \geq 1$,
    \begin{equation*}
        \mathbb{E}_{z}^{\mathbb{Z}^2}\left[ \left(\sum_{s = 0}^{t} \mathds{1}\{X_s = z\} \right)^{2}\right] \le C \log(t)^4.
    \end{equation*}
\end{lem}
\begin{proof}
Define $U_t$ to be the number of times up to time $t$ at which the random walk is at $z$, i.e.
\begin{equation*}
    U_t \coloneqq \sum_{s = 0}^{t} \mathds{1}\{X_s = z\}, \qquad t \in \mathbb{N}.
\end{equation*}
We state the following fact, which dates back to the seminal paper~\cite{erdos1960some} (see (3.11) therein, or~\cite[(5.11)]{dembo2001thick}): For each integer $k\ge1$, for every $\delta \in (0,1)$, $z \in \mathbb{Z}^2$, and $t$ large (depending on $\delta$), 
\begin{equation*}
    \mathbb{P}^{\mathbb{Z}^2}_z\left( U_t > k\log(t)^2 \right) \le t^{-(1-\delta) \pi k}.
\end{equation*}
Then, using this fact, we see that for any $z \in \mathbb{Z}^2$ and $t \geq C(\delta)$, one has
\begin{align*}
    \mathbb{E}^{\mathbb{Z}^2}_{z}\left[ \left(U_t \right)^{2}\right] 
    &= \mathbb{E}^{\mathbb{Z}^2}_{z}\left[ \left(U_t\right)^{2} \mathds{1}\{U_t \le \log(t)^2\}\right] + \mathbb{E}^{\mathbb{Z}^2}_{z}\left[ \left(U_t\right)^{2} \mathds{1}\{U_t > \log(t)^2\}\right]\\
    &\le \log(t)^4 + t^2 t^{-(1-\delta) \pi}.
\end{align*}
Taking for instance $\delta = 1/4$, the desired result readily follows.
\end{proof}
\section{Collisions on combs with fractal skeleton}
\label{sec:Fractal-base}

In this section we focus on the case of a comb graph $G = \mathrm{Comb}(\tilde{G},f_\gamma)$ where $\tilde{G} = (\tilde{V},\tilde{E})$ is some generic graph fulfilling the regularity properties~\eqref{eqn:VolGrowth} and~\eqref{eqn:ResGrowth} below, and for fixed $o \in \tilde{V}$, a radial tooth profile $f_\gamma(v) = \varrho_{p,\gamma}(d_{\tilde{G}}(o,v))$ is specified, where we recall the notation~\eqref{eq:Generic-teeth}. In other words, we consider the graph obtained by adding a segment of length $f_\gamma(v) = \lfloor k^\gamma \rfloor$ to each vertex $v \in \tilde{V}$ such that $d_{\tilde{G}}(o,v) = k$. We now state precisely the aforementioned regularity conditions on $\tilde{G}$.
\begin{equation}\tag{VG}
\label{eqn:VolGrowth}
\begin{minipage}{0.8\linewidth}
    Recall the notation $B_{\tilde{G}}(v, k)$ for the closed $d_{\tilde{G}}$-ball centered at $v \in \tilde{V}$  of radius $k \geq 0$. There exists $\alpha\in [1, 2]$ such that for all $k \ge 0$, one has
    \[
    c k ^{\alpha} \le |B_{\tilde{G}}(v, k)| \le Ck^{\alpha}.
    \]
\end{minipage}\end{equation}
\begin{equation}\label{eqn:ResGrowth}\tag{RG}
    \begin{minipage}{0.8\linewidth}
    For the effective resistance $R^{\tilde{G}}_{\mathrm{eff}}(\cdot, \cdot)$ (recall~\eqref{eq:EffRes}), there exist $2 \le \beta \le \alpha + 1$ with $\beta>\alpha$ such that, for all $u, v \in \tilde{V}$, one has
    \[
    cd_{\tilde{G}}(v, u)^{\beta -\alpha} \le R^{\tilde{G}}_{\mathrm{eff}}(v, u) \le Cd_{\tilde{G}}(v, u)^{\beta -\alpha}.
    \]
    \end{minipage}
    \end{equation}
    We remark that the range chosen for the parameters is classically associated with recurrent fractals, see \cite[Section~1]{BarlowCoulhonKumagai} for a more detailed discussion, and the regularity properties are fulfilled for a wide class of such graphs. A prototypical example to keep in mind is the Sierpi{\'n}ski gasket, for which $\alpha = \log(3)/\log(2)$ and $\beta = \log(5)/\log(2)$. \medskip
     
     We note that for graphs of bounded degree,~\eqref{eqn:VolGrowth} and~\eqref{eqn:ResGrowth} imply a lower bound on the effective resistance between a vertex and the complement of an enclosing ball, namely
    \begin{equation}   
    \label{eq:Eff-res-to-boundary}R^{\tilde{G}}_{\mathrm{eff}}(x,B_{\tilde{G}}(x,r)^c) \geq Cr^{\beta - \alpha}, \qquad \text{for all }x \in \tilde{V}, r \geq 1;
    \end{equation}
    see~\cite[Lemma 2.2]{BarlowCoulhonKumagai}. 
    \begin{obs}
For the purpose of our application to collisions of random walks, it is convenient to work with the set of assumptions \eqref{eqn:VolGrowth} and \eqref{eqn:ResGrowth} on the underlying graph $\tilde{G}$, which imply a uniform estimate on the expected exit time of $d_{\tilde{G}}$-balls by random walks, cf.~\eqref{eqn:ExitCondition} below. This follows by a direct application of~\cite[Proposition~3.4]{BarlowCoulhonKumagai}. 
         The latter condition is 
\begin{equation}
\label{eqn:ExitCondition}\tag{$\mathrm{E}(\beta)$}
    ck^\beta \le \mathbb{E}^{\tilde{G}}_v\left[T_{B_{\tilde{G}}(v,k)}\right] \le Ck^\beta, \qquad \text{for all $v \in \tilde{V}$, $k \in \mathbb{N}$}.
\end{equation}        
There is ample literature on the equivalence of various sets of assumptions to \eqref{eqn:VolGrowth} and \eqref{eqn:ResGrowth} or variants thereof. Although we will not directly use it, we remark that under the mild regularity condition that $G$ is of bounded degree (which our concrete examples satisfy), the uniform sub-Gaussian heat kernel upper and lower bounds~\eqref{eqn:HeatKerUB} and~\eqref{eqn:HeatKerLB} given below are equivalent to~\eqref{eqn:VolGrowth} and \eqref{eqn:ResGrowth} (see~\cite[Section 4]{BarlowCoulhonKumagai} and references therein).
These bounds can be stated as follows:
\begin{equation}\label{eqn:HeatKerUB}\tag{UHK}
    p_{n}(x, y) \le C n^{-\frac{\alpha}{\beta}} \exp\left\{- C\left(\frac{d_{\tilde{G}}(x, y)^\beta}{n \vee 1}\right)^{\frac{1}{\beta- 1}}\right\}, 
\end{equation}
and 
\begin{equation}\label{eqn:HeatKerLB}\tag{LHK}
    p_{n}(x, y) +p_{n+1}(x, y) \ge c n^{-\frac{\alpha}{\beta}} \exp\left\{- c\left(\frac{d_{\tilde{G}}(x, y)^\beta}{n \vee1}\right)^{\frac{1}{\beta- 1}}\right\},
\end{equation}
where both inequalities hold for all $n \ge d_{\tilde{G}}(x, y)$, and $x,y \in \tilde{V}$.
    \end{obs}
The main result of the present section is the following.
\begin{thm}\label{theo:Weak}
    Let $G = \mathrm{Comb}(\tilde{G},f_\gamma)$ be the comb graph where the underlying (base) graph $\tilde{G}$ is of bounded degree and satisfies \eqref{eqn:VolGrowth} and \eqref{eqn:ResGrowth}, 
    and the tooth profile is given by
    \begin{equation}
        f_\gamma(v) = \varrho_{p,\gamma}(d_{\tilde{G}}(o,v)), \qquad \text{with~$\varrho_{p,\gamma}(k) = \lfloor k^\gamma \rfloor$ as in~\eqref{eq:Generic-teeth}.}
    \end{equation}
    Then:
    \begin{enumerate}
        \item If $\gamma \le \beta - \alpha$, the graph $G$ has the infinite collision property.
        \item If $\gamma > \beta - \alpha$, the graph $G$ has the finite collision property.
    \end{enumerate}
\end{thm}
\begin{proof}[Proof of (1)]
We denote each vertex $x\in V$ with two coordinates $x = (x_1, x_2)$, where $x_1$ is the unique point in the vertex set $\tilde{V}$ of the base graph, such that $(x_1,0)$ has minimal distance to $x$ and $x_2 \in \mathbb{N}$ is the ``height'' in the corresponding tooth. We define, for $k \in \mathbb{N}$ and $\tilde{x}\in \tilde{V}$, the sets
\begin{equation}\label{eqn:BallsOnComb}
    D(\tilde{x}, k) \coloneqq \{(x_1, x_2) \in V \colon d_{\tilde{G}}(\tilde{x}, x_1) \le k \}.
\end{equation}
    Statement (1) is a  rather straightforward consequence of~\eqref{eqn:ResGrowth} and the Green kernel criterion~\eqref{eq:Infinite-collision} of~\cite{BPS}, as we briefly explain. Set $o_V = (o,0)$ and define for $r \in \mathbb{N}$ the sets $D_r \coloneqq D(o,r)$ as in~\eqref{eqn:BallsOnComb}. Note that $(D_r)_{r \in \mathbb{N}}$ is an increasing sequence of strict subsets of $V$ with $\bigcup_{r \in \mathbb{N}} D_r = V$. We first argue that 
    \begin{equation}
    \label{eq:Eff-res-equality}
        R^G_{\mathrm{eff}}(o_V,D_r^c) = R_{\mathrm{eff}}^{\tilde{G}}(o,B_{\tilde{G}}(o,r)^c).
    \end{equation}
    Indeed, let $\tilde{h}_r : \tilde{V} \rightarrow \mathbb{R}$ be the minimizer for $R_{\mathrm{eff}}^{\tilde{G}}(o,B_{\tilde{G}}(o,r)^c)^{-1}$ in the definition~\eqref{eq:EffRes}, and $h_r : V \rightarrow  \mathbb{R}$ the corresponding minimizer for  $R^G_{\mathrm{eff}}(o_V,D_r^c)^{-1}$. Note that $h_r$ is constant on each tooth $\{(x_1,x_2) \, : \, 0 \leq x_2 \leq f_\gamma(x_1) \}$, where $x_1 \in B_{\tilde{G}}(o,r)$, since choosing $h_r(x_1,k) \neq h(x_1,0)$ for $1 \leq k\leq f_\gamma(x_1)$ would strictly increase the Dirichlet form. Considering $\tilde{G}$ as a subgraph of $G$ via the inclusion $\tilde{V} \rightarrow V$, $x \mapsto (x,0)$, we thus immediately see that $h_r\vert_{\tilde{V}} = \tilde{h}_r$ for any $r \in \mathbb{N}$, whence~\eqref{eq:Eff-res-equality}. Therefore, it follows that
    \begin{equation}
    \label{eq:Green-lower-first-instance}
        g_{D_r}(o_V,o_V) \stackrel{\eqref{eq:Green-Eff-Res}}{=} R^G_{\mathrm{eff}}(o_V,D_r^c) \stackrel{\eqref{eq:Eff-res-equality}}{=} R_{\mathrm{eff}}^{\tilde{G}}(o,B_{\tilde{G}}(o,r)^c) \stackrel{\eqref{eq:Eff-res-to-boundary}}{\geq} Cr^{\beta-\alpha}.
    \end{equation}
On the other hand, consider $x = (x_1,x_2) \in D_r$ and let $\tilde{I}^{(x_1,r)} \in \mathcal{I}_0(\{x_1\},B_{\tilde{G}}(o,r)^c)$ be the minimizer for $R_{\mathrm{eff}}^{\tilde{G}}(o,B_{\tilde{G}}(o,r)^c)$ in~\eqref{eq:Thompson}. Extending $\tilde{I}^{(x_1,r)}$ to a flow on the graph $G$ by defining $I_{(x_1,k),(x_1,k-1)} = 1$ for $k \in \{1,...,x_2\}$, $I_{(x,0),(y,0)} = \tilde{I}^{(x_1,r)}_{x,y}$ for $x \sim y$ in $\tilde{V}$ (extended antisymmetrically), and $0$ otherwise, the so obtained function $I$ is in $\mathcal{I}_0(\{x\},D_r^c)$ and we obtain
\begin{equation}
\label{eq:Green-upper-first-instance}
    \begin{split}
        g_{D_r}(x,x) & \stackrel{\eqref{eq:Green-Eff-Res}}{=} R_{\mathrm{eff}}^G(x,D_r^c) \stackrel{\eqref{eq:Thompson}}{\leq} E(I,I) \stackrel{\eqref{eq:Thompson}}{=}  x_2 + R_{\mathrm{eff}}^{\tilde{G}}(x_1,B_{\tilde{G}}(o,r)^c) \\ 
        & \stackrel{\eqref{eqn:ResGrowth}}{\leq} r^\gamma + C r^{\alpha - \beta} \stackrel{\gamma \leq \beta - \alpha}{\leq} 2C r^{\beta-\alpha},
    \end{split}
\end{equation}
having also used the monotonicity of the effective resistance in the penultimate step. This shows (1). We show statement (2) in the next subsection.
\end{proof}

\noindent \textbf{Generic notation for comb graphs.} We start by introducing further general notation concerning random walks on combs that will be useful throughout the remainder of the article. 

We let $T_{x, k}$ denote the first time at which the random walk on $G = \mathrm{Comb}(\tilde{G},f_\gamma)$ (with $\gamma > 1$) reaches a horizontal distance of at least $k \in \mathbb{N}$ from the base point of $x = (x_1, 0)\in V$. That is, recalling \eqref{eqn:BallsOnComb}, we let
\begin{equation}\label{eqn:ExitHorBalls}
    T_{x, k} \coloneqq T_{D(x_1,k)} \text{ }(= \inf\{n \ge 0 \colon X_n \in D(x_1, k)^c\}).
\end{equation}
It will be convenient to write $X_n = (\tilde{Z}_n,U_n)$, i.e.~to denote by $(\tilde{Z}_n)_{n \geq 0}$ the projection of the walk $X_n \in V$ to the point in $\tilde{V}$ closest to it, and by $(U_n)_{n \geq 0}$ the ``height'' of the walk on the comb. We define the times $(\zeta_k)_{k \geq 0}$ at which successive horizontal steps are taken by the walk $X$. Formally, we let
\begin{equation}
\begin{split}\label{eqn:DefHorizontalSteps}
    \zeta_0 & \coloneqq 0, \qquad \zeta_1 \coloneqq \inf\{n \geq 1 \, : \, \tilde{Z}_n \neq \tilde{Z}_0 \}, \\ 
    \zeta_{n+1} & \coloneqq \zeta_n + \zeta_1 \circ \theta_{\zeta_n}, \qquad n \geq 1,
    \end{split}
\end{equation}
and set
\begin{equation}\label{eqn:HorizontalSteps}
\begin{split}
    L_{x, k} & \coloneqq \sup\{n \geq 0 \, : \, \zeta_n < T_{x,k} \} \\
    & ( = 
    \text{the number of horizontal steps taken before } T_{x, k}).
    \end{split}
\end{equation}
Informally, by ``horizontal step'' we mean a step of $X$ of the form $(x, 0) \to (y, 0)$ with $x \sim y$ in $\tilde{G}$. Furthermore we set, for any $k \ge 0$, 
\begin{equation}
\label{eq:Horizontal-Walk}
    \tilde{X}_k \coloneqq \tilde{Z}_{\zeta_k}.
\end{equation} 
Heuristically, one may view $\{\tilde{X}_k\}_{k \geq 0}$ simply as the walk $\{\tilde{Z}_k\}_{k \ge 0}$ considered at the (random) times at which the latter jumps. Furthermore, we introduce, for $z \in \tilde{V}$,
    \begin{equation}\label{eqn:Tau}
    \begin{split}
        \tau_z & \coloneqq \zeta_1 \circ \theta_{H_{(z,0)}}  \\
        &= 1 +(\text{time spent in the tooth based at }z \text{ between two consecutive horizontal steps}).
    \end{split}
    \end{equation}
Finally, we set
\begin{equation}\label{eqn:DefVi}
    V_i \coloneqq \tau_{\tilde{X}_i} \circ \theta_{\zeta_i} = \zeta_{i+1} - \zeta_i, \qquad i \in \mathbb{N}.
\end{equation}
To simplify notation, we will abbreviate by $q_t(x, y) \coloneqq p^{G}_t(x,y)$ the heat kernel on the comb graph, for $t \in \mathbb{N}, x,y \in V$.

\subsection{Proof of the finite collision property}

\begin{lem}\cite[Lemma~3.7]{BarlowCoulhonKumagai}\label{lemma:HorHittingTime}
    There exist two positive constants $C, c>0$ such that, for all $x \in V$, $k \ge 1$ and $t>0$, we have
    \begin{equation*}
        \mathbb{P}_{x}\left(L_{x, k} < t\right) \le C \exp\left( -c \left(\frac{k^\beta}{t}\right)^{\frac{1}{\beta - 1}}  \right).
    \end{equation*}
\end{lem}
For completeness and for the benefit of the reader, the proof of this lemma is given in Appendix~\ref{appendix}.

The first consequence of this estimate is the following result on $G$. To ensure that all estimates below are at a correct scale, we introduce the ``effective parameter''
\begin{equation}\label{eqn:Gammaprime}
    \gamma' \coloneqq \gamma \wedge \beta.
\end{equation}
We observe that the regime $\beta- \alpha < \gamma \le \beta$ is the hardest to analyze and, for this regime, one has $\gamma' = \gamma$.
\begin{lem}\label{lemma:ExitCombLDP}
    On the graph $G$ we have that for all $x = (x_1, 0) \in V$, $k \ge 1$ and $t>0$
    \begin{equation*}
        \mathbb{P}_{x}\left(T_{x, k} < t \right) \le C \exp\left( -c \left(\frac{k^{\beta+\gamma'}}{t} \right)^{c(\beta)} \right),
    \end{equation*}
    where $c(\beta)>0$ is a constant depending only on $\beta$.
\end{lem}
\begin{proof}
    Without loss of generality we assume that $x$ is at a horizontal distance larger than $2k/3$ from $o_V$. Indeed, note that otherwise one can apply the strong Markov property at $T_{x, 2k/3}$ and consider the ball $D(y, k/3)$ where $y = X_{T_{x, 2k/3}}$.
   
   We first use that, for any $x = (x_1,0)$ and $\lambda > 0$, one has
   \begin{equation*}
   \begin{split}
       \mathbb{P}_{x}\left(T_{x, k} < t \right) &\le \mathbb{P}_{x}\left(T_{x, k/3} < t \right) \\
       & = \mathbb{P}_{x}\left(T_{x, k/3} < t, L_{x, k/3}< k^\beta/\lambda \right) + \mathbb{P}_{x}\left(T_{x, k/3} < t, L_{x, k/3} \ge k^\beta/\lambda \right).
       \end{split}
   \end{equation*}
   For the first summand we will use Lemma~\ref{lemma:HorHittingTime} and a proper choice of $\lambda$. Let us focus on the second one for now. We define, for a given $y = (y_1, 0)$, $y_1 \in \tilde{V}$ with $d_{\tilde{G}}(o,y_1) > \frac{k}{3}$ and $\nu>0$, the event 
   \begin{equation}\label{eqn:EventBinomial}
   \begin{split}
       &E(y_1, \nu, k) \coloneqq \left\{ \inf\{n \ge 1 \colon \tilde{Z}_n \neq \tilde{Z}_0 = y_1\} \ge \nu^2 k^{2\gamma'} \right\}\\
       &  \ =\{X \text{ takes at least } \nu^2 k^{2\gamma'} \text{ steps in the }y_1 \text{-tooth before the next horizontal step} \}.
   \end{split}
   \end{equation}
   Moreover, we introduce
   \begin{equation}\label{eqn:EventBinomialEnhanced}
   \begin{split}
       \tilde{E}(y_1, \nu, k)  \coloneqq  & \, \{H_{(y_1, \nu k^{\gamma'})} \le \inf\{n \ge 1 \colon \tilde{Z}_n \neq \tilde{Z}_0 = y_1\}\} \\
        &  \cap \{T_{B_G((y_1, \nu k^{\gamma'}), \nu k^{\gamma'})} \circ \theta_{H_{(y_1, \nu k^{\gamma'})}} \ge \nu^2 k^{2\gamma'}\}.
   \end{split}
   \end{equation}
   It is straightforward to see that $E(y_1, \nu, k) \supseteq \tilde{E}(y_1, \nu, k)$. Using standard Gambler's ruin and diffusive estimates for a random walk on a segment of $\mathbb{Z}$ (we remark that a segment of height $k^{\gamma'}$ can always be embedded in one of height $k^\gamma$) we see that, for all $y_1$ as above, if $\nu\le 1/12$, then
   \begin{equation}\label{eqn:ProbBinom}
       \mathbb{P}_{y}\left(\tilde{E}(y_1, \nu, k)\right) \geq \frac{c_1}{\nu k^{\gamma'}}.
   \end{equation}
   Recall the definition \eqref{eqn:DefVi} of $\{V_i\}_{i \ge 0}$, the time spent on the tooth corresponding to the $i$-th vertical excursion, then on the event $\{L_{x,k/3}\geq k^{\beta}/\lambda\}$ we have, $\mathbb{P}_x$-a.s.,
   \[T_{x,k/3}\geq \sum_{i=1}^{L_{x,k/3}}V_i\geq \sum_{i=1}^{k^{\beta}/\lambda}V_i\geq \sum_{i=1}^{k^{\beta}/\lambda}V_i \mathds{1}\{V_i\geq \nu^2 k^{2\gamma'}\}\geq \nu^2 k^{2\gamma'}\sum_{i=1}^{k^{\beta}/\lambda}\mathds{1}\{V_i\geq \nu^2 k^{2\gamma'}\},\]
   whence 
   \begin{align*}
       \mathbb{P}_{x}(T_{x,k/3}<t,L_{x,k/3}\geq k^{\beta}/\lambda)&\leq \mathbb{P}_{x}\Big(\nu^2 k^{2\gamma'}\sum_{i=1}^{k^{\beta}/\lambda}\mathds{1}\{V_i\geq \nu^2 k^{2\gamma'}\}<t\Big)\\
       &=\mathbb{P}_{x}\Big(\sum_{i=1}^{k^{\beta}/\lambda}\mathds{1}\{V_i\geq \nu^2 k^{2\gamma'}\}<\frac{t}{\nu^2 k^{2\gamma'}}\Big).
   \end{align*}
   As we observed below \eqref{eqn:EventBinomialEnhanced}, for each $i = 0, \dots, L_{x, k}$, $\{V_i\geq \nu^2 k^{2\gamma'}\} \supseteq \tilde{E}(\tilde{Z}_i, \nu, k)$. Moreover, the events $\{\tilde{E}(\tilde{X}_{i}, \nu, k)\}_{i \le k^\beta / \lambda}$, are independent under $\mathbb{P}_x(\, \cdot \,  |L_{x, k/3} \geq k^\beta/\lambda)$ by the strong Markov property and the fact that they are independent of the specific height to the tooth considered. Since, on the event $\{L_{x,k/3}\geq k^{\beta}/\lambda\}$, all $V_i, \, 0 \le i \le k^{\beta}/\lambda$ happen in teeth of height larger than $(k/3)^{\gamma'}$, their individual probabilities are bounded below by \eqref{eqn:ProbBinom}, with this observation we obtain
   \begin{equation*}
       \mathbb{P}_{x}\left(T_{x, k/3} < t, L_{x, k/3} \ge k^\beta/\lambda \right) \le \overline{\mathbb{P}}\left(\mathrm{Bin}\left(\frac{k^\beta}{\lambda}, \frac{c_1}{\nu k^{\gamma'}}\right) \le \frac{t}{\nu^2 k^{2\gamma'}} \right)
   \end{equation*}
   (here and in the following, we write $\mathrm{Bin}(M,q)$ for a generic binomially distributed random variable with parameters $\lfloor M \rfloor $ and $q \in [0,1]$, on some auxiliary probability space $(\overline{\Omega},\mathcal{G},\overline{\mathbb{P}})$). 
   We now fix all the parameters of interest: 
   \begin{itemize}
       \item We introduce the quantity $\mu = \frac{t}{k^{\beta+\gamma'}}$.
       \item We fix $\lambda = \mu^{-\frac{1}{3(\beta - 1)}}$.
       \item We fix $\nu = 2 \mu \lambda /c_1 = 2 \mu^{1 - \frac{1}{3(\beta-1)}}/c_1$.
   \end{itemize}
   Before proceeding we make the observation that if $\mu^{1 - \frac{1}{3(\beta-1)}} >  c_1/24$ then one can obtain the result by adjusting the constants in the bound. We also observe that restricting to $\mu^{1 - \frac{1}{3(\beta-1)}} \le  c_1/24$ yields $\nu \le 1/12$ which we required to have enough space in the tooth for the diffusive lower bound to hold. \smallskip

   We also note that, shortening the notation to $\mathrm{Bin}\left(\frac{k^\beta}{\lambda}, \frac{c_1}{\nu k^{\gamma'}}\right) = \mathrm{Bin}$,
   \begin{equation*}
        \frac{t}{\nu^2 k^{2\gamma'}} \frac{1}{\mathbb{E}[\mathrm{Bin}]} = \frac{t \lambda \nu k^{\gamma'}}{\nu^2 k^{2\gamma'} k^{\beta} c_1} = \frac{\lambda t}{ \nu k^{\beta + \gamma'}c_1} = 2 \mu \frac{1}{\mu}.
   \end{equation*}
   Hence, by applying a standard Chernoff bound for binomial random variables we get 
   \begin{align*}
       \mathbb{P}_{x}\left(T_{x, k/3} < t, L_{x, k/3} \ge k^\beta/\lambda \right) \le \exp\left( - c \frac{k^{\beta - \gamma'}}{\lambda \nu}\right) &\le \exp\left( - c \frac{k^{\beta - \gamma'}}{\mu^{-\frac{1}{3(\beta - 1)}+ 1 - \frac{1}{3(\beta-1)}} }\right) \\
       &\le \exp\left( - c \left(\frac{k^{\beta+\gamma'}}{t}\right)^{\frac{3(\beta-1) - 2}{3(\beta-1)}}\right).
   \end{align*}
   Finally, we obtain by Lemma~\ref{lemma:HorHittingTime}
   \begin{equation*}
       \mathbb{P}_{x}\left(L_{x, k/3}< k^\beta/\lambda \right) \le \exp\left( - c\lambda^{\frac{1}{\beta-1}} \right) \le \exp\left( - c \left(\frac{k^{\beta+\gamma'}}{t}\right)^{\frac{1}{3(\beta-1)^2}, } \right).
   \end{equation*}
   We then can take $c(\beta) = \min\{\frac{1}{3(\beta-1)^2}, \frac{3(\beta-1) - 2}{3(\beta-1)} \} > 0$ as $\beta \ge 2$.
\end{proof}
We use this bound to deduce the following result. 

\begin{lem}\label{lemma:HeathonSkeleton}
    For all $t>0$ and all $x = (x_1,0) \in V$ we have
    \begin{equation*}
        q_t(x, x) \le \frac{c}{t^{\frac{\alpha+\gamma'}{\beta+\gamma'}}}.
    \end{equation*}
\end{lem}

\begin{proof}
    Let $k\coloneqq b t^{1/(\beta+\gamma')}$, with $b>0$ to be fixed later. By~\eqref{eqn:ResGrowth} and monotonicity of the effective resistance, we have the upper bound
    \begin{equation*}
        R^{G}_{\mathrm{eff}}(x, D(x_1, k)^c) \le C k^{\beta - \alpha}.
    \end{equation*}
    Then, by the previous lemma we have that
    \begin{equation*}
        \mathbb{P}_{x}\left(T_{x, k} < t \right) \le C \exp\left( -c b^{(\beta+\gamma')c(\beta)} \right).
    \end{equation*}
    Note that the quantity on the right-hand side of the previous display converges to $0$ as $b \to \infty$, hence we can choose $b$ large enough so that it is less than $1/2$. Then applying the formula~\eqref{eq:Bound-HK-via-Green} we find
    \begin{equation*}
        q_t(x, x) \le \frac{2g_{D(x_1,k)}(x, x)}{t \mathbb{P}_{x}(T_{x, k} \geq t)} \le c t^{-1} R^G_{\mathrm{eff}}(x, D(x_1, k)^c) \le c t^{-1}k^{\beta - \alpha} \le c t^{-1}t^{(\beta-\alpha)/(\beta+\gamma')} \le c t^{-\frac{\alpha+\gamma'}{\beta+\gamma'}},
    \end{equation*}
    giving us the desired bound.
\end{proof}

\begin{lem}\label{lemma:HeathonSkeleton2}
    For all $t>0$ and all $x = (x_1,0) \in V$ we have
    \begin{equation*}
        q_t(o_V, x) \le \frac{c}{t^{\frac{\alpha+\gamma'}{\beta+\gamma'}}}.
    \end{equation*}
\end{lem}
\begin{proof}
    Note that for all $t \in [1, k-1] \cap \mathbb{N}$ we trivially have that $q_t(o_V, u) = 0$. On the other hand, using the Cauchy-Schwarz inequality we obtain
    \begin{equation*}
        q_t(o_V, x) \le \sqrt{q_{t}(o_V, o_V)  q_{t}(x, x)} \le \frac{c}{t^{\frac{\alpha+\gamma'}{\beta+\gamma'}}}.
    \end{equation*}
    The second inequality follows directly from Lemma~\ref{lemma:HeathonSkeleton}.
\end{proof}
We finally treat points on the teeth of the comb.
\begin{lem}\label{lemma:HKonFractComb}
    For any $x = (u, h) \in V$, we have
    \begin{equation*}
        q_t(o_V, (u, h)) \le \frac{c}{t^{\frac{\alpha+\gamma'}{\beta+\gamma'}}} e^{-\frac{h^2}{ct}}.
    \end{equation*}
\end{lem}
\begin{proof}
    For $a,b \in \mathbb{N}$, let $T_a^b$ be the exit time of interval $[a, b]$ by a simple random walk on $\mathbb{Z}$. Then from~\cite[(4.11)]{BPS} we have the inequality
    \begin{equation}\label{eqn:GeneralBallot}
        \mathbb{P}^{\mathbb{Z}}_{h} \left(T_0^m = s\right) \le c\frac{h}{s^{3/2}} e^{-h^2/(c's)} +   c\frac{2m-h}{s^{3/2}} e^{-(2m-h)^2/(c's)}, \qquad s \in \mathbb{N}.
    \end{equation}
    Then, by reversibility and setting $k \coloneqq d_{\tilde{G}}(o,u),\delta \coloneqq \frac{\alpha+\gamma'}{\beta+\gamma'}$ we see that
    \begin{equation}
    \label{eq:Application-Ballot-Sec3}
        q_t(o_V, (u, h)) \le \sum_{s= 1}^{t-1} \mathbb{P}^{\mathbb{Z}}_{h}\left(T_0^{k^\gamma} = s\right) q_{t-s}((u, 0), o_V) \le \sum_{s= 1}^{t-1} \mathbb{P}^{\mathbb{Z}}_{h}\left(T_0^{k^\gamma} = s\right) (t-s)^{-\delta},
    \end{equation}
    by Lemma~\ref{lemma:HeathonSkeleton2}. Then,
    \begin{align*}
        q_t((u, h), o_V) &\le hc\left( \sum_{s=1}^{t/2} \frac{1}{s^{3/2}} e^{-h^2/(c's)} t^{-\delta} +  \sum_{s=t/2}^{t-1} \frac{1}{t^{3/2}} e^{-h^2/(c's)} (t-s)^{-\delta} \right)\\
        & + (2k^\gamma - h)c\left( \sum_{s=1}^{t/2} \frac{1}{s^{3/2}} e^{-(2k^\gamma-h)^2/(c's))} t^{-\delta} +   \sum_{s=t/2}^{t-1} \frac{1}{t^{3/2}} e^{-(2k^\gamma-h)^2/(c's)} (t-s)^{-\delta} \right).
    \end{align*}
By standard sum and integral comparisons we obtain
\begin{align*}
        q_t((u, h), o_V) \le hc e^{-h^2/(c't)} \frac{1}{t^\delta \sqrt{t}} + \frac{1}{t^\delta \sqrt{t}}(2k^\gamma - h)ce^{-(2k^\gamma-h)^2/(c't)}.
    \end{align*}
    Furthermore, using that $x e^{-x^2/(ct)} \le \bar{c} \sqrt{t} e^{-x^2/(2ct)}$ for all $x$, if $(2k^\gamma - h)\ge h$ we get
    \begin{equation*}
        q_t((u, h), o_V) \le c t^{-\delta} e^{-h^2/(c't)};
    \end{equation*}
    the result follows by inserting the definition of $\delta$.
\end{proof}

\begin{lem}\label{lemma:SmallTimeFract}
    Let $(u, 0) \in \partial D(o_V,k)$, for $t< k^{\beta + \gamma'}$ we have for all $h>0$
    \begin{equation*}
        q_t(o_V, (u, h)) \le \frac{c}{k^{\alpha + \gamma'}}.
    \end{equation*}
\end{lem}
\begin{proof}
    Let us consider $T_{o_V, k/2}$ (recall~\eqref{eqn:ExitHorBalls}), then
    \begin{equation*}
        \mathbb{P}_{o_V}\left( X_t = x\right) = \mathbb{P}_{o_V}\left( X_t = x , T_{o_V, k/2} \le \frac{t}{2}\right) + \mathbb{P}_{o_V}\left( X_t = x , T_{o_V, k/2} > \frac{t}{2}\right).
    \end{equation*}
    Note that by reversibility
    \begin{equation*}
        \mathbb{P}_{o_V}\left( X_t = x , T_{o_V, k/2} > \frac{t}{2}\right) \le C\mathbb{P}_{x}\left( X_t = o_V , T_{x, k/2} \le \frac{t}{2}\right).
    \end{equation*}
    Hence, we reduce to studying
    \begin{equation*}
        \mathbb{P}_{o_V}\left( X_t = x , T_{o_V, k/2} \le \frac{t}{2}\right) \le \mathbb{P}_{o_V}\left(T_{o_V, k/2} \le \frac{t}{2}\right) \max_{0 \le s \le t/2} \max_{y \in \partial D(o_V, k/2)} \mathbb{P}_{y}\left( X_{t-s} = x\right).
    \end{equation*}
    It follows from Lemma~\ref{lemma:HeathonSkeleton2} that the second probability is bounded from above by $ct^{-\frac{\alpha+\gamma'}{\beta+\gamma'}}$. For the first one we use Lemma~\ref{lemma:HorHittingTime}, and with the choice $t = k^{\beta+\gamma'}/\eta$ we obtain
    \begin{align*}
        \mathbb{P}_{o_V}\left( X_t = x , T_{o_V, k/2} \le \frac{t}{2}\right) 
        &\le c t^{-\frac{\alpha+\gamma'}{\beta+\gamma'}} \exp\left(- c \left( \frac{k^{\beta+\gamma'}}{t}\right)^{c(\beta)}\right)\\
        & \le c k^{-(\alpha+\gamma')} \eta^{\frac{\alpha+\gamma'}{\beta+\gamma'}}e^{-c\eta^{c(\beta)}}\\
        &\le c k^{-(\alpha+\gamma')} \sup_{\eta>0}\eta^{\frac{\alpha+\gamma'}{\beta+\gamma'}}e^{-c\eta^{c(\beta)}} \\
        &\le ck^{-(\alpha+\gamma')}.\qedhere
    \end{align*}
\end{proof}
We now introduce an important exploration procedure of the underlying base graph $\tilde{G} = (\tilde{V},\tilde
{E})$ starting from the origin. To that end, we set
\begin{equation}
\label{eq:Definition-A_k}
\begin{minipage}{0.8\linewidth}
    $A_0 \coloneqq \varnothing$, and define recursively
    \begin{equation*}
    \begin{split}
       \hspace{-1.2cm} A_k & \coloneqq \{u_i \in \tilde{V} \setminus (A_1   \, \cup \, \dots \, \cup  \,A_{k-1} ) \,\colon\, \\
       &  u_i = \arg\min\{d_{\tilde{G}}(o, v) \, \colon v \in \tilde{V} \setminus (A_1   \, \cup \, \dots \, \cup  \,A_{k-1} \cup \{u_1,...,u_{i-1}\} ) \},  \\
       & 1 \leq i \leq \lfloor k^{\alpha-1} \rfloor
     \},  
     \end{split}
    \end{equation*}
    for $k \geq 1$.
\end{minipage}
\end{equation}
   If there are ties one can simply use any deterministic rule. Then we deduce that, for any fixed $\varepsilon>0$, for any $x \in A_k$ we have (for $k$ large enough) $d_{\tilde{G}}(o, x)\ge k^{1-\varepsilon}$. Indeed, if this was not the case, we would obtain $d_{\tilde{G}}(o, x)< k^{1-\varepsilon}$ for some $x \in A_k$, and by the construction in~\eqref{eq:Definition-A_k} of the sets $(A_j)_{j \in \mathbb{N}}$, 
   \begin{equation*}
       \bigcup_{j = 0}^{k-1} A_{j} \subseteq B_{\tilde{G}}(o, (k-1)^{1-\varepsilon}).
   \end{equation*}
   Moreover, the sets $(A_j)_{j \in \mathbb{N}}$ are pairwise disjoint by construction. Hence, using \eqref{eqn:VolGrowth}, we arrive at
\begin{equation*}
    C (k-1)^{\alpha-\varepsilon\alpha} \ge |B_{\tilde{G}}(o, (k-1)^{1-\varepsilon})| \ge \sum_{j = 0}^{k-1} |A_{j}| \ge c(k-1)^{\alpha},
\end{equation*}
which is a contradiction since $C, c>0$ are fixed absolute constants. \smallskip

We now let $\mathcal{Z}_{k, \ell}$ for $k, \ell \in \mathbb{N}$ be the random variable counting the number of collisions happening in $Q_{k, \ell} \coloneqq \{(u, h) \in V \colon u \in A_k, 0 \le h \le \ell\}$ and $\tilde{\mathcal{Z}}_{k, \ell}$ be the random variable counting the number of collisions happening in $\widetilde{Q}_{k, \ell} \coloneqq\{(u, h) \in V \colon u \in A_k, \ell/3 \le h \le 2\ell/3\}$; formally, we set
\begin{equation}
\label{eq:CollInQ-def}
\begin{split}
    \mathcal{Z}_{k,\ell} & = \sum_{n = 0}^\infty \mathds{1}\{n \in \mathbb{N} \, \colon \, X_n = Y_n \in Q_{k,\ell}\}, \\
    \tilde{\mathcal{Z}}_{k,\ell} & = \sum_{n = 0}^\infty \mathds{1}\{n \in \mathbb{N} \, \colon \, X_n = Y_n \in \tilde{Q}_{k,\ell}\}.
    \end{split}
\end{equation}
With the next lemma we bound from above the first moment of $\mathcal{Z}_{k,\ell}$, and bound from below the (conditional) first moment of $\tilde{\mathcal{Z}}_{k,\ell}$, given that $ \mathcal{Z}_{k,\ell}>0$.

\begin{lem}\label{lemma:2ExpectationsFractals}
    For any $\varepsilon>0$ and any $k$ large enough, have that
    \begin{enumerate}
        \item $\mathbb{E}_{o_V}[\mathcal{Z}_{k, \ell}] \le \ell k^{\beta - \alpha - \gamma' - 1 + \varepsilon}$.
        \item $\mathbb{E}_{o_V}[\mathcal{Z}_{k, \ell} | \, \tilde{\mathcal{Z}}_{k, \ell}>0 ] \ge c\ell$.
    \end{enumerate}
\end{lem}
\begin{proof}
    We begin by showing the second statement. Note that, since we are conditioning on the event $\{\tilde{\mathcal{Z}}_{k, \ell}>0\}$, there is a collision at some vertex $x= (u, h) \in V$ with $\ell/3 \le h \le 2\ell/3$. Hence, the total number of collisions in $Q_{k, \ell}$ is bounded from below by the total number of collisions occurring in the tooth $\{(u, h)\colon 0 \le h \le \ell \}$ before this segment is exited by one of the random walks (for the first time). Using \cite[(3.5)]{BPS} we obtain
    \begin{equation*}
        \mathbb{E}_{o_V}[\mathcal{Z}_{k, \ell} | \, \tilde{\mathcal{Z}}_{k, \ell}>0 ] \ge \frac{1}{2} g_{Q_{k, \ell}} (x, x) \ge \frac{1}{2} R^G_{\mathrm{eff}}(x, Q_{k, \ell}^c) \ge c \ell.
    \end{equation*}
    Let us now focus on the first statement. We remark that, in what follows, $\varepsilon>0$ can change from line to line; however, since we can choose it as small as we wish, this will not make a difference. We highlight that, when summing over $Q_{k, \ell}$ the points in $\tilde{V}$ that are the closest to the origin $o$ are at distance $k^{1 - \varepsilon}$. Since the bounds Lemma~\ref{lemma:HKonFractComb} and Lemma~\ref{lemma:SmallTimeFract} are clearly decreasing in $k$, we take the worst case scenario. 
    The first statement follows from Lemma~\ref{lemma:HKonFractComb} and Lemma~\ref{lemma:SmallTimeFract}; indeed, noting that $2\frac{\alpha+\gamma'}{\beta+\gamma'} >1$ since $2 \alpha > \beta$,
    \begin{align*}
        \mathbb{E}_{o_V}[\mathcal{Z}_{k, \ell}] & \stackrel{\eqref{eq:CollInQ-def}}{=} \sum_{t = 0}^\infty \sum_{x \in Q_{k, \ell}} q_{t}(o_V, x)^2 \\
        &\le \ell \sum_{t < k^{\beta+\gamma'}} \frac{ck^{\alpha-1}}{k^{2(\alpha+\gamma') -\varepsilon}} + \ell\sum_{t \ge k^{\beta+\gamma'}} \frac{ck^{\alpha - 1}}{t^{2\frac{\alpha+\gamma'}{\beta+\gamma'}- \varepsilon}}\\
        &\le c\ell k^{\beta -\alpha - \gamma' - 1 + \varepsilon} + \ell k^{\alpha-1} k^{(\beta+\gamma')(-2\frac{\alpha+\gamma'}{\beta+\gamma'}+1 + \varepsilon) }\\
        &\le c\ell k^{\beta -\alpha - \gamma' - 1  + \varepsilon}.\qedhere
    \end{align*}
\end{proof}
We are now ready for the
\begin{proof}[Proof of Theorem~\ref{theo:Weak}]
We have the inequality
\begin{equation}\label{eqn:EasyConditioning}
    \mathbb{E}_{o_V}[\mathcal{Z}_{k, \ell}] \ge \mathbb{E}_{o_V}[\mathcal{Z}_{k, \ell} | \, \tilde{\mathcal{Z}}_{k, \ell}>0 ] \mathbb{P}_{o_V}(\tilde{\mathcal{Z}}_{k, \ell}>0 ),
\end{equation}
which by Lemma~\ref{lemma:2ExpectationsFractals} implies
\begin{equation*}
    \mathbb{P}_{o_V}(\tilde{\mathcal{Z}}_{k, \ell}>0 ) \le c k^{\beta - \alpha - \gamma' - 1 + \varepsilon }.
\end{equation*}
Set $j_0 \coloneqq \inf\{j \in \mathbb{N}\,  \colon \, 2^j \ge k^\gamma\}$, observe that $j_0 \le c \log_{2}(k^\gamma)$ for some absolute $c>0$. Furthermore, set $\mathcal{L} = \mathcal{L}(k) \coloneqq \{2^j\}_{j = 1}^{j_0}$ and note $|\mathcal{L}| = j_0$. We remark that the sets $(\tilde{Q}_{k, \ell})_{k,\ell}$ with $k \ge 0$ and $\ell \in \mathcal{L}(k)$ exhaust the vertex set $V$ of the infinite graph $G$. Then
\begin{equation*}
    \sum_{k = 1}^\infty \sum_{\ell \in \mathcal{L}} \mathbb{P}_{o_V}(\tilde{\mathcal{Z}}_{k, \ell}>0 ) \le \sum_{k = 1}^\infty \sum_{\ell \in \mathcal{L}} c k^{\beta - \alpha - \gamma' - 1 + \varepsilon} \le \sum_{k = 1}^\infty \log_{2}(k^\gamma) c k^{\beta - \alpha - \gamma' -1 + \varepsilon},
\end{equation*}
and this sum is finite whenever $\beta - \alpha - \gamma' -1 + \varepsilon < - 1$, which is the same as $\gamma' > \beta - \alpha + \varepsilon$. For each $\gamma' > \beta - \alpha$ we can choose one $\varepsilon>0$ such that the previous condition is fulfilled. We conclude by applying \cite[Corollary~2.3]{BPS}. We observe that the condition $\gamma' > \beta - \alpha$ coincides with the requirement $\gamma > \beta - \alpha$, whence the desired conclusion follows.
\end{proof}

\section{Collision on the $\mathbb{Z}^2$ comb} 
\label{sec:2D-Comb}In this section we consider the comb graph $\mathrm{Comb}(\tilde{G},f_\gamma)$ in which $\tilde{G} = \mathbb{Z}^2$ and $f_\gamma(z) = \lfloor \log^\gamma(\|z\|_\infty \vee 1)\rfloor$ for $\gamma>0$. Here is our main theorem.

\begin{thm}\label{theo:2d}
    Let $G = \mathrm{Comb}(\mathbb{Z}^2,f_\gamma)$ be the comb graph where the underlying (base) graph is $ \mathbb{Z}^2$ and the tooth profile is given by 
    \begin{equation}
        f_\gamma(z) = \varrho_{\ell,\gamma}(\|z\|_\infty), \qquad \text{with }\varrho_{\ell,\gamma}(k) = \lfloor \log^\gamma(k \vee 1)\rfloor \text{ as in~\eqref{eq:Generic-teeth}}.
    \end{equation}
    We have the following phase transition:
    \begin{enumerate}
        \item If $\gamma \le 1$, then $G$ has the infinite collision property.
        \item If $\gamma > 1$, then $G$ has the finite collision property.
    \end{enumerate}
\end{thm}
As remarked earlier, the first part of the theorem was established as a consequence of~\eqref{eq:Infinite-collision}, in \cite[Section 6, Remark 5]{BPS} and hence, in what follows, we focus on establishing the second statement. To prove (2) we need several estimates concerning the random walk on $\mathbb{Z}^2$ and $\mathrm{Comb}(\mathbb{Z}^2,f_\gamma)$. These estimates will be developed in the next two subsections. 

\subsection{Upper bound for the heat kernel}
We write $q_t(x, y)$ for the heat kernel $p_t^G(x,y)$ of the random walk on the comb graph $G = (V,E) = \mathrm{Comb}(\mathbb{Z}^2,f_\gamma)$, for $t \in \mathbb{N}$ and $x,y\in V$. The following result will be instrumental for the proof of (the second statement in) Theorem~\ref{theo:2d}. 
\begin{prop}\label{prop:HeatKernelZ2}
    Let $x = (z, h)$ with $\|z\|_\infty = k$ and $h \in [0,f_\gamma(k)]\cap \mathbb{N}$ be a vertex in $\mathrm{Comb}(\mathbb{Z}^2,f_\gamma)$.  There is a constant $c>0$ such that the following holds:
    \begin{equation*}
        q_t(0, x) \le 
        \begin{cases}
            \frac{c}{t} & t \ge k^2\log^{\gamma}(k)\\
            \frac{c}{k^2\log^\gamma(k)} & t < k^2\log^{\gamma}(k).
        \end{cases}
    \end{equation*}
\end{prop}
The proof of this proposition is carried out in several intermediate steps. We start with the following lemma.

\begin{lem}\label{lemma:GreenFunctionUpper}
    For all $t \ge 1$ and $x \in V$, we have that
    \begin{equation*}
        q_t(x, x) \le \frac{4}{t} g_{[t/2, t]}(x, x).
    \end{equation*}
\end{lem}
\begin{proof}
    We note that by~\eqref{eq:Heat-kernel-decreasing}, $q_t(x, x)$ is decreasing as a function of $t$ (even), thus
    \begin{equation*}
        q_t(x, x) \le \frac{4}{t} \sum_{s = t/2}^t q_s(x, x) =\frac{4}{t}g_{[t/2, t]}(x, x),
    \end{equation*}
     proving the claim.
\end{proof}

The next lemma is the main ingredient in the proof of Proposition~\ref{prop:HeatKernelZ2}.

\begin{lem}\label{lemma:HeatKernelZ2k} 
    Fix $\varepsilon>0$. There exists a constant $c > 0$ such that, for every $x = (z, 0) \in V$ with $\|z\|_\infty = k$, for all $t> k^{\varepsilon}$, one has
    \begin{equation*}
        q_t(x, x) \le \frac{c}{t}.
    \end{equation*}
\end{lem}

\begin{proof} Throughout the proof we consider a fixed $x = (z,0)$ with $\|z\|_\infty = k$. Our first and main goal is to show that there exists $c>0$ so that, for all $t>k^\varepsilon$, $g_{[t/2, t]}(x, x) \le c$. The main result then immediately follows from this bound and Lemma~\ref{lemma:GreenFunctionUpper}. We use the notation introduced in \eqref{eqn:DefHorizontalSteps} (and below) concerning the time between horizontal steps. Let us define the random variable
\begin{equation}
    \mathrm{Hor}_{n} \coloneqq \sup\{k \ge 0\colon \zeta_k \le n\}, \qquad n \in \mathbb{N},
\end{equation}
which corresponds to the total number of horizontal steps occurring before time $n$. We introduce (for $t \geq 2$) the following (typical) events:
\begin{equation}\label{eqn:LargeDeviationEvents}
        \begin{split}
            A &\coloneqq \{\mathrm{Hor}_{t/2} \ge c_2\frac{t}{2} \log^{-\gamma}(t) \} \\
            &= \left\{X \text{ makes at least } c_2\frac{t}{2} \log^{-\gamma}(t) \text{ horizontal steps before time }\frac{t}{2} \right\}, \text{ and}\\
            B &\coloneqq \{\mathrm{Hor}_{t} \le c_3 t \log^{-\gamma}(t) \} \\
            & = \left\{X \text{ makes at most } c_3 t \log^{-\gamma}(t) \text{ horizontal steps before time }t \right\},
        \end{split}
    \end{equation}
    with $c_2, c_3>0$ two deterministic constants whose values will be fixed later. For any $t \geq 2$ we then use the decomposition
    \begin{equation*}
        \mathbb{E}_{x}\left[ \sum_{s = t/2}^{t} \mathds{1}\{X_s = x\} \right] = \mathbb{E}_{x}\left[ \sum_{s = t/2}^{t} \mathds{1}\{X_s = x\} \mathds{1}\{A\} \right] + \mathbb{E}_{x}\left[ \sum_{s = t/2}^{t} \mathds{1}\{X_s = x\} \mathds{1}\{A^c\} \right].
    \end{equation*} 
    We first focus on the second term, and observe that $\mathds{1}\{A^c\}$ depends only on the first $t/2$ steps of the random walk. On the other hand, $\sum_{s = t/2}^{t} \mathds{1}\{X_s = x\}$ depends only on the position of the walk at times in the interval $[t/2, t] \cap \mathbb{N}$. Hence we find, by the simple Markov property, that
    \begin{equation}
    \label{eq:Bound-Green-fct-A}
        \mathbb{E}_{x}\left[ \sum_{s = t/2}^{t} \mathds{1}\{X_s = x\} \mathds{1}\{A^c\} \right] \le \mathbb{P}_x \left( A^c \right) \sup_{y \in G} \mathbb{E}_{y}\left[ \sum_{s = 0}^{t/2} \mathds{1}\{X_s = x\}\right].
    \end{equation}
    Clearly, by definition,
    \begin{equation}
        \label{eq:Bound-Green-fct-uniform}
        \sup_{y \in G} \mathbb{E}_{y}\left[ \sum_{s = 0}^{t/2} \mathds{1}\{X_s = x\}\right] \leq C \sup_{y \in G} g_{t/2}(y, x),
    \end{equation}
    and using~\eqref{eq:Green-function-maximal}, we have
    $\sup_{y \in G} g_t(y, x) \le g_t(x, x)$. It is not difficult to see that, for some $C>0$
    \begin{equation}
        \label{eq:Bound-Green-Z^2-log}
        g_{t/2}(x, x) \le C R^{\mathbb{Z}^2}_{\mathrm{eff}}(z, B_{\mathbb{Z}^2}(z, t)^c) \le C \log(t).
    \end{equation}
    Using this together with Proposition~\ref{prop:ChernovBound} below (which in particular it states that $\mathbb{P}_x \left( A^c \right)$ is at most $t^{-c'}$ for some $c'>0$) we obtain
    \begin{equation*}
        \mathbb{E}_{x}\left[ \sum_{s = t/2}^{t} \mathds{1}\{X_s = x\} \mathds{1}\{A^c\} \right] \le \mathbb{P}_x \left( A^c \right) \sup_{y \in G} \mathbb{E}_{y}\left[ \sum_{s = 0}^{t/2} \mathds{1}\{X_s = x\}\right] \le \frac{C}{t^c}.
    \end{equation*}
    Hence, we are left with the task of bounding 
    \begin{equation*}
        \mathbb{E}_{x}\left[ \sum_{s = t/2}^{t} \mathds{1}\{X_s = x\} \mathds{1}\{A\} \right] \leq \mathbb{E}_{x}\left[ \sum_{s = t/2}^{t} \mathds{1}\{X_s = x\} \mathds{1}\{A \cap B\} \right] + \mathbb{E}_{z}\left[ \sum_{s = t/2}^{t} \mathds{1}\{X_s = x\} \mathds{1}\{ B^c\} \right].
    \end{equation*}
    We observe that, using the Cauchy-Schwarz inequality
    \begin{equation}\label{eqn:CauchySchwarzErdos}
    \begin{split}
        \mathbb{E}_{x}\left[ \sum_{s = t/2}^{t} \mathds{1}\{X_s = x\} \mathds{1}\{ B^c\} \right] &\le \mathbb{E}_{x}\left[ \left(\sum_{s = t/2}^{t} \mathds{1}\{X_s = x\}\right)^{2} \right]^{1/2}\mathbb{P}_{x}\left(B^c\right)^{1/2} \\
        &\le \frac{C}{t^c} \mathbb{E}_{x}\left[ \left(\sum_{s = t/2}^{t} \mathds{1}\{X_s = x\}\right)^{2} \right]^{1/2},
    \end{split}
    \end{equation}
    again by applying Proposition~\ref{prop:ChernovBound}. For $z \in \mathbb{Z}^2$, we introduce the family of random variables 
    \begin{equation}\label{eqn:Bzed}
    \begin{split}
    B_z &\coloneqq |\{ 0 \le k \le \zeta_1 \, \colon \, X_{k} = z \}|\\
        &=\text{number of visits to }z \text{ between two consecutive horizontal steps}.
    \end{split}
    \end{equation} 
    Note that, under $\mathbb{P}_{(z,0)}$, these random variables are i.i.d.~and follow a $\text{Geom}(4/5)$-distribution on $\mathbb{N} \setminus \{0\}$ (with the exception of at most all $z$ in the box $B_\infty(0,3)$ around the origin). We furthermore consider (possibly upon enlarging the probability space) i.i.d.~copies of $B_z$, denoted for $i \in \mathbb{N}$ by $B^i \coloneqq \{B_{z}^{i}\}_{i \ge 1}$, which are independent of the realization of the horizontal component $\tilde{X}$ of the walk $X$ (recall~\eqref{eq:Horizontal-Walk}). 

    With this we obtain (recalling that  $x = (z,0)$):
    \begin{equation*}
        \mathbb{E}_{x}\left[ \left(\sum_{s = t/2}^{t} \mathds{1}\{X_s = x\}\right)^{2}\right] \le \mathbb{E}_{x}\left[ \left(\sum_{i = 0}^{t} \mathds{1}\{\tilde{X}_i = z\} B_{z}^{i} \right)^{2}\right].
    \end{equation*}
    We remark that, on the right-hand side, $i\in [0, t]$ should be interpreted as the time naturally associated with the horizontal walk $\tilde{X}$ (i.e.~counting only the time steps when a horizontal step is made by $X$). 
    
    Using the independence of $\mathds{1}\{\tilde{X}_i = z\} \text{ and } B_{z}^{i}$ for any $0 \leq i \leq t$, we obtain
    \begin{equation*}
        \mathbb{E}_{x}\left[ \left(\sum_{s = t/2}^{t} \mathds{1}\{X_s = x\}\right)^{2}\right] \le C \mathbb{E}_{x}\left[ \left(\sum_{i = 0}^{t} \mathds{1}\{\tilde{X}_i = z\} \right)^{2}\right] \le C \log(t)^4,
    \end{equation*}
    where we applied Lemma~\ref{lemma:MomentsOccTime}. Upon insertion of the latter display into \eqref{eqn:CauchySchwarzErdos}, we find that the right hand side of the latter is bounded above by $Ct^{-c}$, for two constants $C, c>0$. 
    
    We are now left with the task of bounding from above the quantity
    \[\mathbb{E}_{x}\left[ \sum_{s = t/2}^{t} \mathds{1}\{X_s = x\} \mathds{1}\{A \cap B\} \right].\]
    We have
    \begin{align*}
        \mathbb{E}_{x}\left[ \sum_{s = t/2}^{t} \mathds{1}\{X_s = x\} \mathds{1}\{A \cap B\} \right] &\le \mathbb{E}_{x}\left[ \sum_{i = c_2 \log^{-\gamma}(t)t/2}^{c_3 \log^{-\gamma}(t)t} B^i_z \mathds{1}\{\tilde{X}_i = z\} \right]\\
        &= \mathbb{E}_x\left[\sum_{i = c_2 \log^{-\gamma}(t)t/2}^{c_3 \log^{-\gamma}(t)t} \mathds{1}\{\tilde{X}_i = z\}\right] \mathbb{E}_{x}\left[ B^1_z \right]\\
        &\le Cg_{[t_1, t_2]}^{\mathbb{Z}^2}(z, z) \mathbb{E}_{x}\left[ B^1_z \right],
    \end{align*}
    with the notation $t_1 \coloneqq c_2 \log^{-\gamma}(t)t/2$ and $t_2 \coloneqq c_3 \log^{-\gamma}(t)t$. 
    
We finally observe that, by standard heat kernel estimates for the random walk on $\mathbb{Z}^2$ (see, e.g.,~\cite[Theorem 6.28]{Barbook}), for some $C, c, \hat{c} > 0$, 
\begin{equation*}
    g_{[t_1, t_2]}^{\mathbb{Z}^2}(z, z) \le C \int_{t_1}^{t_2} \frac{c}{s} ds = C\log\left(\frac{c_3 \log^{-\gamma}(t)t}{c_2 \log^{-\gamma}(t)t/2}\right) \le \hat{c}.
\end{equation*}
This concludes the proof.
\end{proof}
\subsection{Concentration inequalities for the time spent in the teeth}
In this subsection, we show that the probability of each of the events $A^c$ and $B^c$, where $A$ and $B$ are defined in~\eqref{eqn:LargeDeviationEvents}, admits a polynomial decay. This is done in Proposition~\ref{prop:ChernovBound} below. We start with a preparatory lemma concerning the time spent in a tooth.
\begin{lem}\label{lemma:MomentsInTooth}
    Consider $z \in \mathbb{Z}^2 \setminus B_\infty(0,3)$ such that $\|z\|_\infty = k$ and recall the notation introduced in \eqref{eqn:Tau}.
    Then there exist constants $c_4, c_5 \in (0, \infty)$ with $c_4<c_5$ such that, for every $y \in V$,
    \begin{equation*}
        c_4 \log^\gamma(k) \le \mathbb{E}_y\left[\tau_z\right] \le c_5 \log^\gamma(k)
    \end{equation*}
    and
    \begin{equation*}
        c_4 \log^{2\gamma}(k) \le \mathbb{E}_y\left[\tau_z^2\right] \le c_5 \log^{4\gamma}(k).
    \end{equation*}
\end{lem}
\begin{proof}
Recall that $X_n = (\tilde{Z}_n, U_n)$ for $n \ge 0$. By the recurrence of $X$ we know that, $\mathbb{P}_y$-almost surely, $H_{(z,0)}<\infty$. 
 Using the strong Markov property at $H_{(z,0)}$ we observe that the random variable $\tau_z$ admits the representation (possibly by enlarging the original probability space to include the random variables $\{\sigma^i_0\}_{i \ge 1}$ described below)
\begin{equation*}
    \tau_z \stackrel{(\mathrm{d})}{=} B_z + \sum_{i = 1}^{B_z-1} \sigma^i_0,
\end{equation*}
where $\stackrel{(\mathrm{d})}{=}$ denotes equality in distribution, $B_z \stackrel{(\mathrm{d})}{=} \mathrm{Geom}(4/5)$ is as in~\eqref{eqn:Bzed}, and $\{\sigma^i_0\}_{i \ge 1}$ are  i.i.d.~random variables independent of $B_z$, equal in law to
\begin{equation}
\label{eq:sigma-0-def}
    \sigma_0 \coloneqq \inf\{ n \geq 1 \, \colon \, U_n = 0 \} \text{ under }\mathbb{P}_{(z,1)}.
\end{equation}
Here, $\sigma_0$ should be interpreted as the time spent in the tooth centered at $z$ until the vertical component $(U_n)_{n \geq 0}$ of the random walk $X$ reaches zero (recall again the notation below~\eqref{eqn:ExitHorBalls}), being counted under $\mathbb{P}_{(z,1)}$ starting after the first step the random walk takes in the corresponding tooth. Note that the process $\{U_{n \wedge \sigma_0} \, : \, n \in \mathbb{N} \}$ under $\mathbb{P}_{(z,1)}$ has the law of a simple random walk under $\mathbb{P}^{\mathcal{I}_k}_1$ on the discrete interval 
    \begin{equation}
    \label{eq:1DTooth-def}
        \mathcal{I}_k \coloneqq [0,\log^{\gamma}(k)]\cap \mathbb{N},
    \end{equation}
stopped at $0$. Moreover, for any $y \in V$ and $z \in \mathbb{Z}^2 \setminus B_\infty(0,3)$, we have that
\begin{equation}\label{eq:ExpLowerBoundteeth}
    \mathbb{E}_y\left[\tau_z\right] = \mathbb{E}_y[B_z] + \mathbb{E}_y[B_z-1] \mathbb{E}_{(z,1)}[\sigma_0] = \mathbb{E}_y[B_z] + \mathbb{E}_y[B_z-1] \mathbb{E}^{\mathcal{I}_k}_1[H_0] = \frac{5}{4} + \frac{1}{4} \mathbb{E}^{\mathcal{I}_k}_1[H_0].
\end{equation}
Using Jensen's inequality we see that
\begin{equation*}
    \mathbb{E}_y\left[\tau_z^2\right] \ge \mathbb{E}_y\left[\tau_z\right]^2 \ge c \mathbb{E}^{\mathcal{I}_k}_1[H_0]^2,
\end{equation*}
and furthermore, using the law of total expectation and the Cauchy-Schwarz inequality (and the fact that the random variables involved are integer-valued) we obtain
\begin{equation*}
    \mathbb{E}_y\left[\tau_z^2\right]  \le \mathbb{E}_y[B_z^2] \mathbb{E}_{(z,1)}[(1+\sigma_0)^2] = \mathbb{E}_y[B_z^2] \mathbb{E}^{\mathcal{I}_k}_1[(1+H_0)^2] \le C \mathbb{E}^{\mathcal{I}_k}_1[H_0^2].
\end{equation*}
Therefore we can restrict our attention to studying the random walk on $\mathcal{I}_k$ started at $1$.

    Using a Gambler's ruin estimate, we further bound from below the expectation on the right-hand side of the equality in~\eqref{eq:ExpLowerBoundteeth} by
    \begin{align*}
     \mathbb{E}_1^{\mathcal{I}_k}\left[H_0\right]&\ge 
        \mathbb{E}^{\mathcal{I}_k}_1[H_0  \mathds{1}\{H_{\log^\gamma(k)} < H_0\} ] \\
        &\geq \frac{1}{\log^{\gamma}(k)}\mathbb{E}^{\mathcal{I}_k}_{\log^{\gamma}(k)}\left[H_0\right]\\
        &\geq \frac{1}{\log^{\gamma}(k)}\mathbb{E}^{\mathcal{I}_k}_{\log^{\gamma}(k)}\left[H_0\mathbbm{1}\{H_0\geq \log^{2\gamma}(k)\}\right]\\
        &\geq \log^{\gamma}(k)\mathbb{P}^{\mathcal{I}_k}_{\log^{\gamma}(k)}(H_0\geq \log^{2\gamma}(k))\geq \frac{\log^{\gamma}(k)}{C}, 
    \end{align*}
    for some constant $C>0$ which does not depend on $k$. This establishes the desired lower bound on both the first and second moment of $H_0$ under $\mathbb{P}^{\mathcal{I}_k}_1$ and hence of $\tau_z$ under $\mathbb{E}_y$.
    Now we move to the upper bounds, and start with that of $H_0^2$. Note that, for a positive integer $L=L(k)$ to be specified later, we can write
    \[
    \mathbb{E}_1^{\mathcal{I}_k}\left[H_0^2\right]\leq L^2 + \mathbb{E}_1^{\mathcal{I}_k}\left[H_0^2\mathbbm{1}_{\{H_0^2\geq L^2\}}\right]
       \leq 2L^2+\sum_{j> L}\mathbb{P}_1^{\mathcal{I}_k}(H_0>j),
    \]
    where for the last inequality we have used that, for a non-negative, integer-valued random variable $Z$ under a probability measure $\mathbb{P}$ and $h\in \mathbb{N}$, it holds that
\begin{equation}\label{trickpernonegativa}
        \mathbb{E}[Z\mathbbm{1}_{\{Z\geq h\}}]=h\mathbb{P}(Z>h)+\sum_{j>h}\mathbb{P}(Z>j).
    \end{equation}
    Assume that $L\geq \log^{2\gamma}(k)$ and set (ignoring the dependence on $k$ in the notation) $A_j\coloneqq j\log^{-2\gamma}(k)$ (as usual, we keep flooring and ceiling implicit in the notation).
    Clearly, 
    \[
    \mathbb{P}_{1}^{\mathcal{I}_k}(H_0>j)\leq \mathbb{P}_{1}^{\mathcal{I}_k}(H_0>A_j\log^{2\gamma}(k)).
    \]
    We now split the time window $\{1,\dots, A_j\log^{2\gamma}(k)\}$ under consideration into $A_j$ sub-intervals (only possibly overlapping on their respective start- and endpoints), each of length $\log^{2\gamma}(k)$: 
    \[
    \{1,\dots, A_j\log^{2\gamma}(k)\}=\bigcup_{i=1}^{A_j}\{(i-1)\log^{2\gamma}(k),\dots, i\log^{2\gamma}(k)\}\eqqcolon \bigcup_{i=1}^{A_j}I_i.
    \]
    To avoid confusion with the walk $X$ on the comb we denote the canonical process on $\mathcal{I}_k$ or $\mathbb{Z}$ by $W$.
    With this, we see that $W$ stays above zero for $A_j\log^{2\gamma}(k)$ steps if, and only if, it stays above zero in each $I_i$ ($1 \leq i \leq A_j$). Writing $\mathcal{F}_h = \sigma(W_n \, \colon \, 0 \leq n \leq h)$ for $h \in \mathbb{N}$, we find 
\begin{equation}
\begin{split}\label{eq:iterate}
        \mathbb{P}_{1}^{\mathcal{I}_k}&(H_0>A_j\log^{2\gamma}(k))=\mathbb{P}_{1}^{\mathcal{I}_k}\Big(\bigcap_{i=1}^{A_j}\{W_n>0 \text{ }\forall n\in I_i\}\Big)\\
        &=\mathbb{E}_{1}^{\mathcal{I}_k}\Big[\mathds{1}\Big\{\bigcap_{i=1}^{A_j-1}\{W_n>0 \text{ }\forall n\in I_i\}\Big\}\mathbb{P}_{1}^{\mathcal{I}_k}\Big(W_n>0 \text{ }\forall n\in I_{A_j}\mid \mathcal{F}_{(A_j-1)\log^{2\gamma}(k)}\Big)\Big].
        \end{split}
    \end{equation}
    In view of the fact that $W$ is more likely to remain positive when starting from $\log^{\gamma}(k)$ than when starting from $i \in \{1,\dots,\log^\gamma(k)-1\}$,
    we see that $\mathbb{P}_1^{\mathcal{I}_k}$-almost surely,
    \begin{align*}
        \mathbb{P}_{1}^{\mathcal{I}_k}\Big(W_n>0 \text{ }\forall n\in I_{A_j}\mid \mathcal{F}_{(A_j-1)\log^{2\gamma}(k)}\Big)&\leq \mathbb{P}^{\mathcal{I}_k}_{\log^{\gamma}(k)}\Big(W_n>0 \text{ }\forall 1\leq n\leq \log^{2\gamma}(k) \Big)\\
        &\leq \mathbb{P}^{\mathbb{Z}}_{\log^{\gamma}(k)}\Big(W_n>0 \text{ }\forall 1\leq n\leq \log^{2\gamma}(k) \Big)\\
        &=\mathbb{P}_{0}^{\mathbb{Z}}\Big(W_n>-\log^{\gamma}(k) \text{ }\forall 1\leq n\leq \log^{2\gamma}(k) \Big).
    \end{align*}
    Using the diffusivity of $W$, we now argue that the last probability is at most $\varepsilon$ for some constant $\varepsilon \in (0,1)$. Indeed, by the local limit theorem on $\mathbb{Z}$ we readily obtain
    \begin{align*}
        \mathbb{P}_{0}^{\mathbb{Z}}\Big(\exists n\in \{1,\dots,\log^{2\gamma}(k)\}: W_n\leq -\log^{\gamma}(k)\Big)&\geq \mathbb{P}_{0}^{\mathbb{Z}}\Big(W_{\log^{2\gamma}(k)}\in ( -2\log^{\gamma}(k),-\log^{\gamma}(k)]\Big)\\
        &=\sum_{h=-2\log^{\gamma}(k)+1}^{-\log^{\gamma}(k)}\mathbb{P}_{0}^{\mathbb{Z}}\Big(W_{\log^{2\gamma}(k)}=h)\\
        &\geq \sum_{h=-2\log^{\gamma}(k)+1}^{-\log^{\gamma}(k)} \frac{c}{\log^{\gamma}(k)}=c,
        \end{align*}
        for some constant $c\in (0,1)$, whence 
        \[\mathbb{P}_{0}^{\mathbb{Z}}\Big(W_n>-\log^{\gamma}(k) \text{ }\forall 1\leq n\leq \log^{2\gamma}(k) \Big)\leq 1-c\eqqcolon \varepsilon.\]
        Inserting the last estimate back into~\eqref{eq:iterate} we obtain
        \[
        \mathbb{P}_{1}^{\mathcal{I}_k}\Big(\bigcap_{i=1}^{A_j}\{W_n>0 \text{ }\forall n\in I_i\}\Big)\leq \varepsilon \mathbb{P}_{1}^{\mathcal{I}_k}\Big(\bigcap_{i=1}^{A_j-1}\{W_n>0 \text{ }\forall n\in I_i\}\Big),
        \]
        and iterating we conclude that
        \[\mathbb{P}_{1}^{\mathcal{I}_k}(H_0>A_j\log^{2\gamma}(k))\leq \exp(-A_j\log(1/\varepsilon)).\]
        Therefore we arrive at
        \begin{align*}
            \mathbb{E}_1^{\mathcal{I}_k}\left[H_0^2\right]&\leq 2L^2+\sum_{j> L}\exp(-A_j\log(1/\varepsilon))\\
            &\leq 2L^2+\sum_{j\geq L}\exp(-j\log^{-2\gamma}(k)\log(1/\varepsilon))\leq 3L^2.
        \end{align*}
        Taking $L=\log^{2\gamma}(k)$ yields the desired upper bound. \medskip

        Concerning the first moment there remains to evaluate the average length of each ``excursion'' $H_0$. For $a, b \in \mathcal{I}_k$ let $H(a \leftrightarrow b) = H_b + H_a \circ \theta_{H_b}$ denote the commute time between the two vertices. We have
        \begin{equation*}
            \mathbb{E}_1^{\mathcal{I}_k}[H_0]\le\mathbb{E}_1^{\mathcal{I}_k}[H(1 \leftrightarrow 0)].
        \end{equation*}
        Applying the commute time identity (see \cite[Proposition~10.7]{LPBook}) we obtain
        \begin{equation*}
            \mathbb{E}_1^{\mathcal{I}_k}[H(1 \leftrightarrow 0)] \le R_{\mathrm{eff}}^{\mathcal{I}_k}(1, 0) \times 2 |\mathcal{I}_k| \le 3 \log^\gamma(k).
        \end{equation*}
        The proof is concluded.
    \end{proof}

\begin{prop}\label{prop:ChernovBound}
    Consider the events defined in \eqref{eqn:LargeDeviationEvents}, there exist constants $C, c>0$ such that, for all $x = (z, 0)$ such that $\|z\|_\infty = k \ge 4$ and $t \ge k^\varepsilon$ 
    \begin{equation}
    \label{eq:Bound-A-c}
        \mathbb{P}_x\left( A^c \right)\le \frac{C}{t^c},
    \end{equation}
    and
    \begin{equation}
    \label{eq:Bound-B-c}
        \mathbb{P}_x\left( B^c \right)\le \frac{C}{t^c}.
    \end{equation}
\end{prop}

\begin{proof}
    We recall the definition given in \eqref{eqn:DefVi}, that is $V_i = \tau_{\tilde{X}_i} \circ \theta_{\zeta_i} = \zeta_{i+1} - \zeta_i, \, i \in \mathbb{N}$, where we recall that $(\tilde{X}_i)_{i \in \mathbb{N}}$ denotes the projection of the walk on $\tilde{V}$, counting only horizontal steps (see~\eqref{eq:Horizontal-Walk}). 
    
    Note that conditionally on the trajectory $\{\tilde{X}_{i}\}_{i = 0}^r$ the random variables $\{V_i\}_{i =0}^r$ are independent for any $r \geq 1$. Moreover, conditionally on $\tilde{X}_{i} = z$, $z\in\mathbb{Z}^2$, we have $V_i \stackrel{(\mathrm{d})}{=} \tau_z$ with the former as in \eqref{eqn:Tau}. Note in passing that the random variables $\{V_i\}_{i = 0}^r$ are not necessarily identically distributed, but their distributions are heavily correlated (since the teeth of a certain height are surrounded by other teeth of similar height).\\

    \noindent \textbf{The event $A^c$:} We begin by proving~\eqref{eq:Bound-A-c}. Informally, we aim at controlling the probability that the random walk $X$ makes ``too few'' horizontal steps within its first $t/2$ steps. The latter corresponds to an upward deviation for the sum of the excursion times in the teeth. More precisely, as we explain below, we have
    \begin{align}\label{eqn:InversionEventA}
        \mathbb{P}_x\left(A^c\right) \le \mathbb{P}_x\left( \sum_{j = 0}^{c_2 \log^{-\gamma}(t) t/2 } V_j \ge \frac{t}{4}\right).
    \end{align}
    Indeed, the inequality $\sum_{j = 0}^{c_2 \log^{-\gamma}(t) t/2 } V_j < \frac{t}{4}$ implies that the random walk $X$ has performed at least $c_2 \log^{-\gamma}(t) t/2$ horizontal steps in its first $t/2$ steps as, by definition, there is always a horizontal step between $V_j$ and $V_{j+1}$, and $c_2 \log^{-\gamma}(t) t/2  < t/4$ for all $t$ large. This means that
    \begin{equation*}
        \left\{\sum_{j = 0}^{c_2 \log^{-\gamma}(t) t/2 }V_j < \frac{t}{4} \right\} \subseteq A,
    \end{equation*}
    and the inequality \eqref{eqn:InversionEventA} follows readily. \medskip

    We now derive an upper bound on the probability on the right-hand side of~\eqref{eqn:InversionEventA}. To that end, we condition the random variables $\{V_j\}_{j = 0}^{r(A)}$, where $r(A) \coloneqq c_2 \log^{-\gamma}(t) t/2$, on a realization of the horizontal walk $\tilde{X}$. More precisely, let $\omega = \{\omega_i\}_{i = 0}^{r(A)}$ be any fixed path of points in $\mathbb{Z}^2$ such that $z = \omega_0 \sim \omega_1 \sim \dots \sim \omega_{r(A)}$. We use the shorthand notation $\{\tilde{X}_{[0,r(A)]} =\omega\}$ for the event $\{\tilde{X}_0 = \omega_0, \tilde{X}_1 = \omega_1, \dots, \tilde{X}_{r(A)} = \omega_{r(A)}\}$. As already observed, conditionally on $\{\tilde{X}_{[0,r(A)]} =\omega\}$ the $\{V_j\}_{j =0}^{r(A)}$ are independent (but not identically distributed) random variables with expectation
    \begin{equation}
    \label{eq:Expectation-V-Variables}
        \mathbb{E}_{x}\left[V_j | \tilde{X}_{[0,r(A)]} = \omega\right] = \mathbb{E}_{(\omega_j,0)}\left[\tau_{\omega_j}\right], \qquad j = 0, \dots, r(A).
    \end{equation}
    We will use the shorthand notation $\mathrm{E}_{\omega}\left[V_j\right]$ to denote the quantity $\mathbb{E}_{x}[V_j | \tilde{X}_{[0,r(A)]} = \omega]$ (similar notation applies to different moments). The reader might find useful to think about this procedure in the following way: the realization of $\tilde{X}$ is a random environment and the sum of $\{V_j\}_{j \ge 0}$ is a random walk in random environment. In this picture, $\mathrm{E}_{\omega}\left[V_j\right]$ may be viewed as the quenched mean of the $j$-th increment of the random walk. \medskip
    
    We now prove that for all $t \geq k^\varepsilon$ and paths $\omega$ starting from $z$, one has (deterministically)
    \begin{equation}
        0 \le \mathrm{E}_{\omega}\left[V_j\right]\le c^*\log^{\gamma}(t), \qquad \text{for all }j = 0, \dots, r(A).
    \end{equation}
    for some $c^*(c_2) >0$ such that $a \mapsto c^\ast(a)$ is increasing. To see this, we note that for any such path $z = \omega_0 \sim \dots \sim \omega_{r(A)}$,
    \begin{equation}
        \max\{\|\omega_j\|_\infty \, : \, j \in \{0,\dots ,r(A) \} \} \leq k^\varepsilon + c_2 t/2 \log^{-\gamma}(t)  \stackrel{t \geq k^\varepsilon}{\leq} (1 + c_2)t,
    \end{equation}
    and the claim follows by Lemma~\ref{lemma:MomentsInTooth} and~\eqref{eq:Expectation-V-Variables}. \medskip

    Therefore, we see that
    \begin{align*}
        & \mathbb{P}_x\left( \sum_{j = 0}^{c_2 \log^{-\gamma}(t) t/2 } V_j \ge \frac{t}{4}\right) \\
        &\le \sum_{\omega} \mathbb{P}_x\left( \sum_{j = 0}^{c_2 \log^{-\gamma}(t) t/2 } (V_j - \mathrm{E}_{\omega}\left[V_j\right]) \ge \frac{t}{4} - c^*c_2 t/2 \mid \tilde{X}_{[0,r(A)]}  = \omega \right)\mathbb{P}_x(\tilde{X}_{[0,r(A)]}=\omega),
    \end{align*}
    with the sum ranging over the set of all paths $\omega$ in $\mathbb{Z}^2$ of length $r(A)$ with $\omega_0 = x$. 
    We can choose $c_2$ such that $c^*c_2 \le 1/4$ (recall that decreasing $c_2$ also decreases $c^*$). Then, for any $\omega$ as above, we obtain
    \begin{equation}
    \begin{split}
    \label{eq:Bound-A-event}
        \mathbb{P}_x&\left( \sum_{j = 0}^{c_2 \log^{-\gamma}(t) t/2 } (V_j - \mathrm{E}_{\omega}\left[V_j\right]) \ge \frac{t}{4} - c^*c_2 t/2 \mid \tilde{X}_{[0,r(A)]} = \omega \right) 
        \\&\le \mathbb{P}_x\left( \sum_{j = 0}^{c_2 \log^{-\gamma}(t) t/2 } (V_j - \mathrm{E}_{\omega}\left[V_j\right]) \ge \frac{t}{8} \mid \tilde{X}_{[0,r(A)]} = \omega \right)\le C \frac{t \log^{-\gamma}(t) \log^{4\gamma}(t)}{t^2}.
        \end{split}
    \end{equation}
    In the last step we used Chebyshev's inequality and Lemma~\ref{lemma:MomentsInTooth}. Note that the inequality thus holds for $\mathbb{P}_x$-almost every realization of $\tilde{X}$, and taking expectations with respect to $\mathbb{E}_x$ yields the result and concludes the proof of~\eqref{eq:Bound-A-c}.
    \\

    \noindent \textbf{The event $B^c$:} We now turn to the upper bound on the probability of $B^c$ in~\eqref{eq:Bound-B-c}. As we explain below, we use a similar type of argument as for~\eqref{eqn:InversionEventA}, to derive
    \begin{equation}\label{eqn:InversionEventB}
        \mathbb{P}_x\left(B^c\right) \le \mathbb{P}_x\left( \sum_{j = 0}^{c_3 \log^{-\gamma}(t) t} V_j \le 2t\right).
    \end{equation}
    Indeed, if the time spent in the first $c_3 \log^{-\gamma}(t) t$ vertical excursions exceeds $2t$ then there are less than $c_3 \log^{-\gamma}(t) t$ horizontal steps up to time $t$. More precisely, one has
\begin{equation*}
    \left\{\sum_{j = 0}^{c_3 \log^{-\gamma}(t) t}V_j > 2t \right\} \subseteq B,
\end{equation*}
proving~\eqref{eqn:InversionEventB}. 
    For $Z \subseteq \mathbb{Z}^2$ and integers $0 \le a < b$, we denote by $
        \ell_Z(a,b) = \sum_{k = a}^b \mathbbm{1}\{\tilde{X}_k \in Z \}$    the occupation time accumulated by the horizontal component $(\tilde{X}_k)_{k \geq 0}$ of the random walk in the set $Z$ between times $a$ and $b$, counting time steps again only when horizontal steps are made. Then, we see
    \begin{align*}
        \mathbb{P}_x\left( \sum_{j = 0}^{c_3 \log^{-\gamma}(t) t} V_j \le 2t\right) &\le \mathbb{P}_x\left( \sum_{j = 0}^{c_3 \log^{-\gamma}(t) t} V_j \le 2t, \ell_{B_\infty(0, t^{1/3})}(0, c_3 \log^{-\gamma}(t) t) \le \frac{1}{2} c_3 \log^{-\gamma}(t) t\right) \\
        &+ \mathbb{P}_x\left( \ell_{B_\infty(0, t^{1/3})}(0, c_3 \log^{-\gamma}(t) t) > \frac{1}{2} c_3 \log^{-\gamma}(t) t\right).
    \end{align*}
We start by bounding the second summand using a first moment computation. To that end, note that for any $x' = (z',0)$ with $z'\in \mathbb{Z}^2$,
\begin{equation}
\label{eq:Green-Occupation-Z^2}
    \mathbb{E}_{x'}\left[ \ell_{B_\infty(0, t^{1/3})}(0, c_3 \log^{-\gamma}(t) t)\right] = \sum_{y \in B_\infty(0,t^{1/3})} g^{\mathbb{Z}^2}_{ c_3 \log^{-\gamma}(t) t}(z',y) \leq Ct^{2/3}\log(t),
\end{equation}
similarly as in~\eqref{eq:Bound-Green-fct-uniform} and~\eqref{eq:Bound-Green-Z^2-log}. With this, we infer that
\begin{align*}
    \mathbb{P}_x\left( \ell_{B_\infty(0, t^{1/3})}(0, c_3 \log^{-\gamma}(t) t) > \frac{1}{2} c_3 \log^{-\gamma}(t) t\right) &\le C \frac{\mathbb{E}_x\left[ \ell_{B_\infty(0, t^{1/3})}(0, c_3 \log^{-\gamma}(t) t)\right]}{t\log^{-\gamma}(t)}\\
    & \stackrel{\eqref{eq:Green-Occupation-Z^2}}{\le} C \frac{t^{2/3}\log(t)}{t\log^{-\gamma}(t)} \le C t^{-c}.
\end{align*}
Before proceeding with the other quantity, we will employ again the strategy we used for the event $A^c$ to condition on the event $\{\tilde{X}_{[0,r(A)]} = \omega\}$. We consider the set of paths $\omega = \{\omega
_i\}_{i = 0}^{r(B)}$ in $\mathbb{Z}^2$ with $z = \omega_0 \sim \omega_1 \sim \dots \sim \omega_{r(B)}$, with $r(B) \coloneqq c_3 \log^{-\gamma}(t) t$ (this definition coincides with the one before, but with $r(B)$ replacing $r(A)$). Since the event $\{\ell_{B_\infty(0, t^{1/3})}(0, c_3 \log^{-\gamma}(t) t) \le \frac{1}{2} c_3 \log^{-\gamma}(t) t\}$ depends only on $(\tilde{X}_i)_{i = 0}^{r(B)}$, we obtain
\begin{equation}\label{eqn:TimeOutside}
    \begin{split}
        & \mathbb{P}_x\left( \sum_{j = 0}^{c_3 \log^{-\gamma}(t) t} V_j \le 2t, \ell_{B_\infty(0, t^{1/3})}(0, c_3 \log^{-\gamma}(t) t) \le \frac{1}{2} c_3 \log^{-\gamma}(t) t\right) \\
        &= \mathbb{E}_x\left[  \mathbb{P}_x\left( \sum_{j = 0}^{c_3 \log^{-\gamma}(t) t} V_j \leq 2t \mid \tilde{X}_{[0,r(B)]} = \omega \right) \mathds{1}\{\ell_{B_\infty(0, t^{1/3})}(0, c_3 \log^{-\gamma}(t) t) \le \frac{1}{2} c_3 \log^{-\gamma}(t) t\}\right].
    \end{split}
\end{equation}
It follows that the set of $\omega$ contributing to the conditional probability can be restricted to paths $\omega = \{\omega_i\}_{i = 0}^{r(B)}$ such that $\mathcal{J}(\omega) = \{ i \in \{0, \dots, r(B)\} \, \colon \,  \omega_i \not \in B_\infty(0, t^{1/3})\}$ satisfies $|\mathcal{J}(\omega)| \ge r(B)/2$, and we set
\begin{equation}
 \Omega_* =  \{\omega_0 = z \sim \dots \sim \omega_{r(B)} \, \colon \, |\mathcal{J}(\omega)| \geq r(B)/2 \}.
\end{equation}
Given $\omega \in \Omega_\ast$, we now define for $1 \leq j \leq |\mathcal{J}(\omega)|$ the random variable $V_j^{\mathrm{out}}$ as the $j$-th vertical excursion in a tooth based outside the $\mathbb{Z}^2$-ball $B_\infty(0, t^{1/3})$, i.e.~more formally setting 
\begin{equation}
    \begin{split}
        i_1 & = \inf\{i \in \{0,\dots, r(B)\} \, \colon \omega_i \notin B_\infty(0,t^{1/3}) \}, \\
        i_k & = \inf\{i \in \{0,\dots, r(B)\} \setminus \{i_1,\dots,i_{k-1} \} \, \colon \omega_i \notin B_\infty(0,t^{1/3}) \}, k \in \{2, \dots, |\mathcal{J}(\omega)|\},
    \end{split}
\end{equation}
and $V^{\mathrm{out}}_j = V_{i_j}$. Conditionally on $\{\tilde{X}_{[0,r(B)]} = \omega\}$ with $\omega \in \Omega_*$, we have that $V_j^{\mathrm{out}}$ is distributed as $\tau_{\omega_{i_j}}$ with $\omega_{i_j} \notin B_\infty(0,t^{1/3})$.  With slight abuse of notation, we will write $\mathrm{E}_{\omega}[V^{\mathrm{out}}_j]$ for the mean of $V_j^{\mathrm{out}}$ conditionally on $\{\tilde{X}_{[0,r(B)]} = \omega\}$. Equipped with these observations, we have that for any $\omega \in \Omega_*$
\begin{equation}\label{eqn:Quenched}
    \mathbb{P}_x\left( \sum_{j = 0}^{c_3 \log^{-\gamma}(t) t} V_j 
    \leq 2t \mid \tilde{X}_{[0,r(B)]} = \omega \right) \le \mathbb{P}_x\left( \sum_{j = 1}^{c_3/2 \log^{-\gamma}(t) t} V_j^{\mathrm{out}} \le 2t \mid \tilde{X}_{[0,r(B)]} = \omega \right).
\end{equation}
From the latter, we deduce that the left-hand side of~\eqref{eqn:TimeOutside} is bounded above by
\begin{equation}\label{eqn:InsertLast}
    \mathbb{E}_x \left[ \mathbb{P}_x\left( \sum_{j = 1}^{c_3/2 \log^{-\gamma}(t) t} V_j^{\mathrm{out}} \le 2t \mid \tilde{X}_{[0,r(B)]} = \omega \right)\mathds{1}\{\ell_{B(0, t^{1/3})}(0, c_3 \log^{-\gamma}(t) t) \le \frac{t}{2} c_3 \log^{-\gamma}(t) \}\right].
\end{equation}
Recall that $\{V_j^{\mathrm{out}}\}_{j = 1}^{r(B)}$ are independent and, on any realization of the walk $\{\tilde{X}_{[0,r(B)]} = \omega\}$ with $\omega \in \Omega_*$, they are such that $c_6\log^{\gamma}(t) \le \mathrm{E}_{\omega}[V^{\mathrm{out}}_j] \le c^*\log^{\gamma}(t)$ with $c_6 = (1/3)^\gamma$. Moreover, we also see that $\mathrm{E}_{\omega}[(V^{\mathrm{out}}_j)^2] \le c_{**} \log^{4\gamma}(t)$ by Lemma~\ref{lemma:2Expectations2d}. \medskip
    
We now derive an upper bound on the right-hand side of~\eqref{eqn:Quenched}. Consider the event $\{\tilde{X}_{[0, r(B)]} = \omega\}$ for a fixed $\omega \in \Omega_*$. Multiplying both sides of the inequality inside the probability by $-\lambda$ for some $\lambda>0$ to be determined, taking $\exp(\cdot)$ on both sides and using Markov's inequality results in the upper bound
\[
e^{2\lambda t}\mathbb{E}_{x}\Big[\exp\Big(-\lambda\sum_{j = 1}^{c_3/2 \log^{-\gamma}(t) t} V_j^{\mathrm{out}}\Big) \mid \tilde{X}_{[0,r(B)]} = \omega \Big]=e^{2\lambda t}\prod_{j=1}^{c_3/2 \log^{-\gamma}(t) t}\mathrm{E}_{\omega}\Big[\exp\Big(-\lambda V_j^{\mathrm{out}} \Big)\Big].\]
Since $e^{-x}\leq 1-x+x^2/2$ for every $x\geq 0$, we observe that
\begin{equation}
    \begin{split}
\mathrm{E}_{\omega}\Big[\exp\Big(-\lambda V_j^{\mathrm{out}} \Big)\Big]& \leq 1-\lambda \mathrm{E}_{\omega}[V_j^{\mathrm{out}}]+\frac{\lambda^2}{2}\mathrm{E}_{\omega}\left[(V_j^{\mathrm{out}})^2\right]\\
&\leq \exp\Big(-\lambda \mathrm{E}_{\omega}[V_j^{\mathrm{out}}]+\frac{\lambda^2}{2}\mathrm{E}_{\omega}\left[(V_j^{\mathrm{out}})^2\right]\Big),
\end{split}
\end{equation}
where for the last inequality, we have used that $e^x\geq 1+x$ for every $x \in \mathbb{R}$. Thus, we arrive at
\begin{align*}
    \prod_{j=1}^{c_3/2 \log^{-\gamma}(t) t}\mathrm{E}_{\omega}\Big[\exp\Big(-\lambda V_j^{\mathrm{out}} \Big)\Big]&\leq \prod_{j=1}^{c_3/2 \log^{-\gamma}(t) t}\exp\Big(-\lambda \mathrm{E}_{\omega}[V_j^{\mathrm{out}}]+\frac{\lambda^2}{2}\mathrm{E}_{\omega}[(V_j^{\mathrm{out}})^2]\Big)\\
    &=\exp\Bigg(-\lambda \sum_{j=1}^{c_3/2 \log^{-\gamma}(t) t}\mathrm{E}_{\omega}[V_j^{\mathrm{out}}]+\frac{\lambda^2}{2}\sum_{j=1}^{c_3 \log^{-\gamma}(t) t}\mathrm{E}_{\omega}[(V_j^{\mathrm{out}})^2]\Bigg).
\end{align*}
Note that, for any realization $\{\tilde{X}_{[0,r(B)]} = \omega\}$ with $\omega \in \Omega_*$, we see that $\lambda \sum_{j=1}^{c_3/2 \log^{-\gamma}(t) t}\mathrm{E}_{\omega}[V_j^{\mathrm{out}}] \ge 2^{-1}\lambda  c_3 c_6 t$. We finally fix the constant $c_3$ so that $2^{-1}c_3c_6 \ge 3$, and thus
\begin{equation*}
    \exp\Bigg(2\lambda  t - \lambda \sum_{j=1}^{c_3/2 \log^{-\gamma}(t) t}\mathrm{E}_{\omega}[V_j^{\mathrm{out}}]\Bigg) \le \exp(2\lambda  t - 2^{-1}\lambda  c_3 c_6 t) \le \exp(-\lambda t).
\end{equation*}
We also notice that 
\begin{equation*}
    \sum_{j=1}^{c_3 \log^{-\gamma}(t) t}\mathrm{E}_{\omega}[(V_j^{\mathrm{out}})^2] \le c_{**} t \log^{4\gamma}(t).
\end{equation*}
Collecting the pieces, we have left for $\omega \in \Omega_*$
\begin{equation*}
    \mathbb{P}_x\left( \sum_{j = 1}^{c_3/2 \log^{-\gamma}(t) t} V_j^{\mathrm{out}} \le 2t \mid \tilde{X}_{[0,r(B)]} = \omega \right) \le \exp\left(-\lambda t + \frac{\lambda^2}{2} c_{**} t \log^{4\gamma}(t) \right) \le e^{- c t^{1/3}}.
\end{equation*}
The final inequality is justified by setting $\lambda = t^{-1/2}$, $c > 0$ is some absolute constant. We insert this last display in~\eqref{eqn:InsertLast}, this yields the result.
\end{proof}
 
\subsection{Collecting the estimates}

\begin{lem} \label{lemma:HKGaussZ2}
Let $x = (z, h)$ with $\|z\|_\infty = k, h \ge 0$, there exist two positive constants $c, c'>0$ such that for all $t \ge 0$, 
\begin{equation*}
    q_t(0, x) \le \frac{c}{t} e^{ - \frac{h^2}{c't}}.
\end{equation*}
\end{lem}

\begin{proof}
    The proof is very similar to the one of Lemma~\ref{lemma:HKonFractComb}, and we only provide details for the steps in which adaptation to the present set-up is required. Similarly to~\eqref{eq:Application-Ballot-Sec3}, we obtain
    \begin{equation*}
        q_t(0, (z, h)) \le \sum_{s= 1}^{t-1} \mathbb{P}^{\mathbb{Z}}_{h}\left(T_0^{\log^\gamma(k)} = s\right) q_{t-s}((z, 0), 0) \le \sum_{s= 1}^{t-1} \mathbb{P}^{\mathbb{Z}}_{h}\left(T_0^{\log^\gamma(k)} = s\right) (t-s)^{-1}.
    \end{equation*}
    For all $t>k^{1/2}$, we claim that there exists $c>0$ such that
    \begin{equation*}
        \mathbb{P}^{\mathbb{Z}}_{h}\left(T_0^{\log^\gamma(k)} \ge t/2\right) \le \exp\left(- c\frac{t}{\log^{2\gamma}(t)}\right).
    \end{equation*}
    Indeed, by standard random walk estimates there exists $c'>0$ such that
    \begin{equation*}
        \inf_{\ell \in [0, \log^\gamma(k)] }\mathbb{P}^{\mathbb{Z}}_{\ell}\left(T_0^{\log^\gamma(k)} \le \log^{2\gamma}(k)\right) \ge c'.
    \end{equation*}
    Then, by the strong Markov property we find
    \begin{align*}
        \mathbb{P}^{\mathbb{Z}}_{h}\left(T_0^{\log^\gamma(k)} \ge t/2\right) 
        &\le \prod_{j = 1}^{t/\log^{2\gamma}(k)} \left(1 - \inf_{\ell \in [0, \log^\gamma(k)] }\mathbb{P}^{\mathbb{Z}}_{\ell}\left(T_0^{\log^\gamma(k)} \le \log^{2\gamma}(k)\right) \right)\\
        & \le \left(1 - c'\right)^{t/\log^{2\gamma}(k)}\\
        &\le \exp\left(- c\frac{t}{\log^{2\gamma}(t)}\right),
    \end{align*}
    as $\log^{2\gamma}(k) \le \frac{1}{2}\log^{2\gamma}(t)$ in the range considered.
    Then,
    \begin{equation*}
        \sum_{s= t/2}^{t-1} \mathbb{P}^{\mathbb{Z}}_{h}\left(T_0^{\log^\gamma(k)} = s\right) (t-s)^{-1} \le t^2 \exp\left(- c\frac{t}{\log^{2\gamma}(t)}\right) \le \frac{c}{t} e^{ - \frac{h^2}{c't}}.
    \end{equation*}
    We are now left with the task of bounding
    \begin{align*}
        \sum_{s= 1}^{t-1} \mathbb{P}^{\mathbb{Z}}_{h}\left(T_0^{\log^\gamma(k)} = s\right) (t-s)^{-1} .
    \end{align*}
    Employing \eqref{eqn:GeneralBallot}, we obtain for some $C>0$
    \begin{align*}
        q_t(0, (z, h)) &\le C \left(h\sum_{s = 1}^{t/2} \frac{1}{s^{3/2}} e^{-h^2/(c's)} \frac{1}{t} + (2\log^\gamma(k)-h)\sum_{s=1}^{t/2} \frac{1}{s^{3/2}} e^{-(2\log^\gamma(k)-h)^2/(c's))} \frac{1}{t}\right)\\
        &\le C h e^{-h^2/(c't)}\frac{1}{t^{3/2}} + C (2\log^\gamma(k)-h) e^{-(2\log^\gamma(k)-h)^2/(c't)}\frac{1}{t^{3/2}}.
    \end{align*}
    Proceeding as in the last steps of the proof of Lemma~\ref{lemma:HKonFractComb} yields the desired result.
\end{proof}

\begin{lem} \label{lemma:SmallTimeZ2}
Let $x = (z, h)$ with $\|z\|_\infty = k, h\ge0$, there exist a positive constant $c>0$ such that for all $t \le k^2\log^\gamma(k)$, 
\begin{equation*}
    q_t(0, x) \le \frac{c}{k^2\log^\gamma(k)}.
\end{equation*}
\end{lem}

\begin{proof}
    The proof follows closely the one of Lemma~\ref{lemma:SmallTimeFract}, let us go through the necessary adaptations. We need to bound
    \begin{equation*}
        \mathbb{P}_{0}\left( X_t = x , T_{0, k/2} \le \frac{t}{2}\right) \le \mathbb{P}_{0}\left(T_{0, k/2} \le \frac{t}{2}\right) \max_{0 \le s \le t/2} \max_{y \in \partial B(0, k/2)} \mathbb{P}_{y}\left( X_{t-s} = x\right).
    \end{equation*}
    We apply Lemma~\ref{lemma:HorExitCombZ2} to the first term on the right hand side and Lemma~\ref{lemma:HeatZ2CS} to the second term.
    Fixing $t = k^{2}\log^\gamma(k)/\eta$ we obtain
    \begin{align*}
        \mathbb{P}_{0}\left( X_t = x , T_{0, k/2} \le \frac{t}{2}\right) 
        &\le c t^{-1} \exp\left( -c \left(\frac{k^2\log^\gamma(k)}{t}\right)^{1/3}  \right)\\
        & \le c k^{-2}\log^{-\gamma}(k) \eta e^{-c\eta^{1/3}}\\
        &\le c k^{-2}\log^{-\gamma}(k) \sup_{\eta>0}\eta e^{-c\eta^{1/3}} \\
        &\le ck^{-2}\log^{-\gamma}(k).
    \end{align*}
    This concludes the proof.
\end{proof}

\begin{lem}\label{lemma:HorExitCombZ2}
    Let $x = (z, 0)$ with $\|z\|_\infty = k$, $k \ge 1$ and $t>k^\varepsilon$ 
    \begin{equation*}
        \mathbb{P}_{x}\left(T_{x, k} < t\right) \le C \exp\left( -c \left(\frac{k^2\log^\gamma(k)}{t}\right)^{1/3}  \right).
    \end{equation*}
\end{lem}
\begin{proof}
    The proof follows the same lines as the one of Lemma~\ref{lemma:ExitCombLDP}, we will only describe the necessary adaptations. \medskip

    To that end, we consider
   \begin{equation*}
       \mathbb{P}_{x}\left(T_{x, k/3} < t, L_{x, k/3} \ge k^2/\lambda \right) \le \overline{\mathbb{P}}\left(\mathrm{Bin}\left(\frac{k^2}{\lambda}, \frac{c_1}{\nu \log^\gamma(k)}\right) \le \frac{t}{\nu^2 \log^{2\gamma}(k)} \right),
   \end{equation*}
   where $\overline{\mathbb{P}}$ is defined as in the proof of Lemma~\ref{lemma:ExitCombLDP}. We now fix all the parameters of interest: 
   \begin{itemize}
       \item Fix $\mu = \frac{t}{k^{2}\log^\gamma(k)}$.
       \item Fix $\lambda = \mu^{-\frac{1}{3}}$.
       \item Fix $\nu = 2 \mu \lambda /c_1 = 2 \mu^{2/3}/c_1$.
   \end{itemize}
   We observe
   \begin{equation*}
        \frac{t}{\nu^2 \log^{2\gamma}(k)} \frac{1}{\mathbb{E}[\mathrm{Bin}]} = \frac{t \lambda \nu \log^{\gamma}(k)}{\nu^2 \log^{2\gamma}(k)k^{2} c_1} = \frac{\lambda t}{ \nu k^{2} \log^{\gamma}(k)c_1} = 2 \mu \frac{1}{\mu}.
   \end{equation*}
   The rest of the proof proceeds along the exact same lines of Lemma~\ref{lemma:ExitCombLDP} with the new parameters in place and using Lemma~\ref{lemma:HorHittingTime} with $\beta = 2$. Indeed, the latter holds by a standard martingale argument for $(\|\tilde{X}_n\|^2 - n)_{n \geq 0}$
   since $L_{x,k}$ has the same law under $\mathbb{P}_x$ as $T_{B_{\mathbb{Z}^2}}(x,k)$ under $\mathbb{P}^{\mathbb{Z}^2}_x$.
\end{proof}

\begin{lem}\label{lemma:HeatZ2CS}
    For all $t>0$ and all $x = (z, 0) $ such that $ \|z\|_\infty = k$ we have
    \begin{equation*}
        q_t(0, x) \le \frac{c}{t}.
    \end{equation*}
\end{lem}
\begin{proof}
    Note that for all $t \in [1, k-1] \cap \mathbb{N}$ we trivially have $q_t(0, x) = 0$. On the other hand, the Cauchy-Schwarz inequality yields
    \begin{equation*}
        q_t(0, x) \le \sqrt{q_{t}(0, 0)  q_{t}(x, x)} \le \frac{c}{t}.
    \end{equation*}
    Note that the second inequality follows directly from Lemma~\ref{lemma:HeatKernelZ2k}.
\end{proof}

\subsection{Conclusion}

 We define the random variable $\mathcal{Z}_{k, \ell}$ to count the number of collisions happening in $Q_{k, \ell} \coloneqq \{(u, h) \in V \colon \|u\|_\infty = k, 0 \le h \le \ell\}$ and $\tilde{\mathcal{Z}}_{k, \ell}$ to count the number of collisions happening in $\widetilde{Q}_{k, \ell} \coloneqq \{(u, \ell) \in V \colon \|u\|_\infty = k, \ell/3 \le h \le 2\ell/3\}$, see~\eqref{eq:CollInQ-def} for the precise definition. 

\begin{lem}\label{lemma:2Expectations2d}
    For any $\varepsilon>0$ and any $k$ large enough, have that
    \begin{enumerate}
        \item $\mathbb{E}_{0}[\mathcal{Z}_{k, \ell}] \le \ell k^{-1}\log^{-\gamma}(k)$.
        \item $\mathbb{E}_{0}[\mathcal{Z}_{k, \ell} | \, \tilde{\mathcal{Z}}_{k, \ell}>0 ] \ge c\ell$.
    \end{enumerate}
\end{lem}
\begin{proof}
    The second statement follows in the same way as its counterpart in Lemma~\ref{lemma:2ExpectationsFractals}.
    
    Regarding the first statement, applying Lemma~\ref{lemma:HKGaussZ2} and Lemma~\ref{lemma:SmallTimeZ2}, we write 
    \begin{align*}
        \mathbb{E}_{0}[\mathcal{Z}_{k, \ell}] &= \sum_{t = 0}^\infty \sum_{x \in Q_{k, \ell}} q_{t}(0, x)^2 \\
        &\le \ell \sum_{t < k^2 \log^\gamma(k)} \frac{ck}{k^{4} \log^{2\gamma}(k)} + \ell\sum_{t \ge k^2 \log^\gamma(k)} \frac{ck}{t^{2}}\\
        &\le c\ell k^{-1}\log^{-\gamma}(k) + c\ell k k^{-2}\log^{-\gamma}(k) \\
        &\le c\ell k^{-1}\log^{-\gamma}(k),
    \end{align*}
    where the second inequality follows from a standard comparison of a series to an integral.
\end{proof}

\begin{proof}[Proof of Theorem~\ref{theo:2d}]
We focus on the case $\gamma>1$, and prove statement (2), i.e.~the finite collision property for $\mathrm{Comb}(\mathbb{Z}^2,f_\gamma)$ for the aforementioned regime. Recall the computation \eqref{eqn:EasyConditioning}, by Lemma~\ref{lemma:2Expectations2d} this implies
\begin{equation*}
    \mathbb{P}_{0}(\tilde{\mathcal{Z}}_{k, \ell}>0 ) \le c k^{-1}\log^{-\gamma}(k).
\end{equation*}
Set $j_0 \coloneqq \inf\{j \in \mathbb{N}\,  \colon \, 2^j \ge \log^\gamma(k)\}$, observe that $j_0 \le c \log_{2}(\log^\gamma(k))$ for some absolute $c>0$. Furthermore, set $\mathcal{L} = \mathcal{L}(k) \coloneqq \{2^j\}_{j = 1}^{j_0}$ and note $|\mathcal{L}| = j_0$. We remark that the sets $(\tilde{Q}_{k, \ell})_{k,\ell}$ with $k \ge 0$ and $\ell \in \mathcal{L}(k)$ exhaust the vertex set $V$ of the infinite graph $G$. Then
\begin{equation*}
    \sum_{k = 1}^\infty \sum_{\ell \in \mathcal{L}} \mathbb{P}_{0}(\tilde{\mathcal{Z}}_{k, \ell}>0 ) \le \sum_{k = 1}^\infty \sum_{\ell \in \mathcal{L}} c k^{-1}\log^{-\gamma}(k) \le \sum_{k = 1}^\infty \log_{2}(\log^\gamma(k)) c k^{-1}\log^{-\gamma}(k).
\end{equation*}
For arbitrarily small $\varepsilon>0$ we have that $\log_{2}(\log^\gamma(k)) c k^{-1}\log^{-\gamma}(k) \le c k^{-1}\log^{-\gamma + \varepsilon}(k)$ for all large enough $k$. We choose $\varepsilon$ in such a way that $\gamma - \varepsilon > 1$ (this is possible since $\gamma>1$). Hence, since the series with general term $k^{-1}\log^{-s}(k)$ is integrable whenever $s>1$, we conclude that
\begin{equation*}
    \sum_{k = 1}^\infty \log_{2}(\log^\gamma(k)) c k^{-1}\log^{-\gamma}(k) < \infty.
\end{equation*}
An application of \cite[Corollary~2.3]{BPS} finishes the proof.
\end{proof}

\section{Collisions on combs over supercritical percolation clusters}

\label{sec:Random-base-graphs}

In this section, we prove that the phase transition for a comb with a base graph given by a typical realization of a supercritical Bernoulli bond percolation cluster containing the origin occurs at the same scale as for $\mathbb{Z}^2$. Our main result appears in Theorem~\ref{theo:perc} below. \medskip 

We begin by introducing some notation and useful results on Bernoulli bond percolation that will be used in the proof of Theorem~\ref{theo:perc}. Let $E(\mathbb{Z}^2)$ denote the set of nearest-neighbor edges in $\mathbb{Z}^2$ and let $p \in [0,1]$. We consider the probability measure $\mathbb{Q}$ on $\Omega_0 = \{0,1 \}^{E(\mathbb{Z}^2)}$ such that the canonical coordinates $(\mu_e)_{e \in  E(\mathbb{Z}^2)}$ are i.i.d.~Bernoulli($p$)-random variables. We consider the random subgraph of $(\mathbb{Z}^2,E(\mathbb{Z}^2))$ obtained by removing all edges $e \in E(\mathbb{Z}^2)$ with $\mu_e = 0$, that is $(\mathbb{Z}^2,\{e \in E(\mathbb{Z}^2) \, : \, \mu_e = 1 \})$, and we refer to the connected components of the latter as \textit{clusters}. For $\varnothing \neq A \subseteq \mathbb{Z}^2$ we let $E(A)$ stand for the set of edges $e \in E(\mathbb{Z}^2)$ with both endpoints in $A$, and call the connected components of $(A,\{e \in E(A) \, : \, \mu_e =1 \})$ the \textit{clusters in $A$}. Throughout this entire section, we work with a fixed
\begin{equation}
\label{eq:Supercritical}
    p \in \left(\tfrac{1}{2} ,1\right],
\end{equation}
and generic positive constants may implicitly depend on $p$. Under the standing assumption~\eqref{eq:Supercritical} it is well known that 
\begin{equation}
    \begin{minipage}{0.8\linewidth}
      there exists a $\mathbb{Q}$-a.s.~unique cluster of infinite cardinality, which will be denoted by $\mathcal{C}_\infty$, 
    \end{minipage}
\end{equation}
see~\cite{grimmett1999percolation}. It is convenient to ``center'' the percolation cluster at the origin, and to this end we note that for $p > \frac{1}{2}$, one has $\mathbb{Q}(0 \in \mathcal{C}_\infty) > 0$, and therefore we can consider the probability measure $\mathbb{Q}_0  = \mathbb{Q}(\, \cdot \,|0 \in \mathcal{C}_\infty)$. \medskip

We start by collecting some additional controls from the literature concerning effective resistances and heat kernels on $\mathcal{C}_\infty$. To begin with, we note that by~\cite[Proposition 4.3]{boivin2013existence}, there exists a set $\Omega^{(1)} \subseteq \{0 \in \mathcal{C}_\infty\}$ of full $\mathbb{Q}_0$-probability, such that 
    \begin{equation}
    \label{eq:BoivinRau-bound}
    \begin{minipage}{0.8\textwidth}
for every $\mu \in \Omega^{(1)} $, there exists $N_0(\mu) < \infty$ such that $R^{\mathcal{C}_\infty}_{\mathrm{eff}}(0,B_{\mathcal{C}_\infty}(0,N)^c) \geq c_7 \log(N)$,
for $N \geq N_0(\mu)$.
 \end{minipage}
    \end{equation}
We will also need a corresponding uniform upper bound on the effective resistance. To that end, we define for $N \in \mathbb{N}$ and $\mu \in \{0 \in \mathcal{C}_\infty\}$ the set
\begin{equation}
    \mathcal{C}^N \coloneqq \text{ the largest cluster in $B_\infty(0,N)$},
\end{equation}
where we recall that $B_\infty(0,N) = [-N,N] \cap \mathbb{Z}^2$ (and breaking ties by some arbitrary rule). By~\cite[Theorem 1.1]{abe2015effective}, we know that there exists a set $\Omega^{(2)} \subseteq \{0 \in \mathcal{C}_\infty\}$ of full $\mathbb{Q}_0$-probability such that
 \begin{equation}
 \label{eq:Abe-bound}
\begin{minipage}{0.8\textwidth}
 for every $\mu \in \Omega^{(2)}$, there exists $N_0'(\mu) < \infty$ such that 
        $\sup_{x,y \in \mathcal{C}^N} R^{\mathcal{C}_\infty}_{\mathrm{eff}}(x,y) \leq c_8 \log(N)$, for every $N \geq N_0'(\mu)$.
\end{minipage}
     \end{equation}
In our application, we want to use both~\eqref{eq:BoivinRau-bound} and~\eqref{eq:Abe-bound} together, and we need to rule out the (very) unlikely event that conditioned on $\{0 \in \mathcal{C}_\infty\}$, the largest cluster $\mathcal{C}^N$ in $B_\infty(0,N)$ is not the cluster at the origin. 
We denote by $\mathcal{C}_{B_\infty(0,N)}(x)$ the cluster of $x \in B_\infty(0,N)$ in $B_\infty(0,N)$.

The following result is well-known, and we include a short proof for completeness.
\begin{lem}
\label{lem:Box-origin-lemma}
    Let $p > \frac{1}{2}$. There exists a set $\Omega^{(3)} \subseteq \{0 \in \mathcal{C}_\infty\}$ of full $\mathbb{Q}_0$-probability such that
    \begin{equation}
    \begin{minipage}{0.8\textwidth}
for every $\mu \in \Omega^{(3)}$, there exists $N''_0(\mu) < \infty$ such that for every $N \geq N_0''(\mu)$ one has $\mathcal{C}^N =
        \mathcal{C}_{B_\infty(0,N)}(0)$.
    \end{minipage}
    \end{equation}
\end{lem}
\begin{proof}
Note that on the event $\{0 \in \mathcal{C}_\infty\}$ the origin is connected to the boundary $\partial_{\mathbb{Z}^2} B_\infty(0,N)$, and thus the connected component of the origin fulfills $|\mathcal{C}(0)| \geq N$.  We denote by $\mathcal{C}^{\mathrm{d},N}$ the largest connected component in $B_\infty(0,N) + (\frac{1}{2},\frac{1}{2})$ of the subgraph of the dual lattice $(\mathbb{Z}^2+(\frac{1}{2},\frac{1}{2}),E(\mathbb{Z}^2+(\frac{1}{2},\frac{1}{2})))$ obtained by removing edges $f$ for which $\mu_{e(f)} =1$, where $e(f)$ is the unique edge in $E(\mathbb{Z}^2)$ crossing $f$, see~\cite[Section 11.2]{grimmett1999percolation}. Therefore we have
\begin{equation}
    \begin{split}
        \mathbb{Q}_0(\mathcal{C}^N \neq
        \mathcal{C}_{B_\infty(0,N)}(0)) 
        & \leq \frac{\mathbb{Q}(\text{$B_\infty(0,N)$ contains two disjoint cluster of size $\geq N$} )}{\mathbb{Q}(0 \in \mathcal{C}_\infty)} 
        \\
        & \leq \frac{\mathbb{Q}(|\mathcal{C}^{\mathrm{d},N}| \geq c\sqrt{N})}{\mathbb{Q}(0 \in \mathcal{C}_\infty)} \leq c(p) N^2\exp(-c'(p)\sqrt{N}),
    \end{split}
\end{equation}
by~\cite[Theorem (6.75)]{grimmett1999percolation} and a union bound, since the percolation induced on the dual lattice has parameter $1-p < \frac{1}{2}$. The result then follows by a Borel-Cantelli argument.
\end{proof}

\medskip

Next, we state some pertinent quenched controls on the heat kernel $p_n^{\mathcal{C}_\infty}$ on the infinite cluster $\mathcal{C}_\infty$, which are derived in~\cite{barlow2004random}, and will be instrumental for the proof of Theorem~\ref{theo:perc} below. By~\cite[Theorem 1]{barlow2004random} and (0.5) of the same reference, there exists a set $\Omega^{(4)} \subseteq \{0 \in \mathcal{C}_\infty\}$ of full $\mathbb{Q}_0$-probability and random variables $(\vartheta_x)_{x \in \mathbb{Z}^d}$ with 
\begin{equation}
\label{eq:Theta-finite}
\vartheta_x < \infty, \qquad \text{ for }\mu \in \Omega^{(4)}, x \in \mathcal{C}_\infty,
\end{equation}
and for every $\mu \in \Omega^{(4)}$, $x,y \in \mathcal{C}_\infty$, and $n \geq \vartheta_x(\mu) \vee \|x-y\|_1$ one has the bound
\begin{equation}
\label{eq:HKUB}
p_{n}^{\mathcal{C}_\infty}(x,y)  \leq \begin{cases}
     \frac{c}{n} \exp\left(-\frac{c'\|x-y\|_1^2}{n} \right), & \text{if $n$ and $\|x-y\|_1$ are both odd or both even}, \\
     0, & \text{else}
     \end{cases}
\end{equation}
(we use the convention $\vartheta_x(\mu) = \infty$ if $\mu \notin \Omega^{(4)}$, or $\mu \in \Omega^{(4)}$ but $x \notin \mathcal{C}_\infty$). The  $\vartheta_x$, $x \in \mathbb{Z}^d$, are equal in law and one has the bound
\begin{equation}
\label{eq:Integrability-Barlow}
\mathbb{Q}(x \in \mathcal{C}_\infty, \vartheta_x \geq n) \leq c \exp\left(-c' n^{c_9} \right).
\end{equation}
By standard heat kernel lower bounds on $\mathbb{Z}^2$, cf.~\cite[Theorem 6.28]{Barbook}, we therefore have for every $\mu \in \Omega^{(4)}$, $x,y \in \mathcal{C}_\infty$, and $n \geq \vartheta_x(\mu) \vee \|x-y\|_1$ the bound
\begin{equation}
\label{eq:HKUB-compared}
p_{n}^{\mathcal{C}_\infty}(x,y)  \leq Cp_n^{\mathbb{Z}^2}(x,y).
\end{equation}
Going forward, we need good controls on the range of times for which~\eqref{eq:HKUB-compared} is valid. To that end, define the events (for $k \in \mathbb{N}$)
    \begin{equation}
    \label{eq:Bad-events-Heat-kernel}
        \mathcal{A}_k \coloneqq \bigcup_{\substack{x \in \mathbb{Z}^2, \\ \|x\|_\infty = k}} \{x\in \mathcal{C}_\infty,
        \vartheta_x > (\log k)^{2/c_9} \},
    \end{equation}
    and note that by~\eqref{eq:Integrability-Barlow}, one has
    \begin{equation}
        \sum_{k = 1}^\infty \mathbb{Q}(\mathcal{A}_k) \leq Ck \exp\left(- c' \log^2 (k) \right) < \infty.
    \end{equation}
    By the Borel-Cantelli lemma, we therefore see that for an event $\Omega^{(5)} \subseteq \{0 \in \mathcal{C}_\infty\}$ of full $\mathbb{Q}_0$-probability, we have that for every $\mu \in \Omega^{(5)}$ there exists $K_0(\mu)< \infty
    $ such that
    \begin{equation}
    \label{eq:Theta-small}
        \vartheta_x \leq (\log(k))^{\frac{2}{c_9}}, \qquad \text{for all } x \in \mathcal{C}_\infty \text{ with }\|x\|_\infty = k, \text{ and } k \geq K_0(\mu).
    \end{equation}

We now show a result in the spirit of Lemma~\ref{lemma:MomentsOccTime} for the infinite cluster $\mathcal{C}_\infty$. We highlight that we do not expect the estimate in the next lemma to be sharp, but it will be sufficient for our purposes.

\begin{lem}\label{lemma:MomentsOccTimePerc}
    Let $p > \frac{1}{2}$ and $\mu \in \Omega^{(4)} \cap \Omega^{(5)}$ (with $\Omega^{(4)}$ and $\Omega^{(5)}$ chosen above~\eqref{eq:Theta-finite} and~\eqref{eq:Theta-small}, respectively). Consider the the simple random walk $X$ on $\mathcal{C}_\infty$ under  $P_z^{\mathcal{C}_\infty}$, $z \in \mathcal{C}_\infty$.  For any $\varepsilon, \eta\in (0, 1)$ and $t \geq k^\varepsilon$, one has
    \begin{equation*}
        \mathbb{E}^{\mathcal{C}_\infty}_{x}\left[ \left(\sum_{s = 0}^{t} \mathds{1}\{X_s = x\} \right)^{2}\right] \le C t^{2\eta}, \qquad \text{ for all }x\in \mathcal{C}_\infty, \|x\|_\infty = k \text{ or } x = 0,
    \end{equation*}
    whenever $k \geq K_1(\mu,\eta,\varepsilon) \vee \vartheta_0^{1/(\varepsilon\eta)}(\mu)$.
\end{lem}
\begin{proof} We first consider the case $x\in \mathcal{C}_\infty, \|x\|_\infty = k$.
 Let $\varepsilon,\eta \in (0,1)$ be fixed, then we see that by~\eqref{eq:Theta-small}, and using that $t \geq k^\varepsilon$,
 \begin{equation}\label{eqn:ParametersLargeEnough}
        t^\eta  \geq k^{\varepsilon\eta} > (\log(k))^{\frac{2}{c_9}} \geq \vartheta_x, \qquad \text{for all } x \in \mathcal{C}_\infty \text{ with }\|x\|_\infty = k, k \geq K_0(\mu) \vee C(\eta,\varepsilon).
    \end{equation}
We then consider $t\geq k^\varepsilon$ with $k \geq K_1(\mu,\eta,\varepsilon) \coloneqq K_0(\mu) \vee C(\eta,\varepsilon)$ as in \eqref{eqn:ParametersLargeEnough} and write
\begin{equation*}
\begin{split}
    \mathbb{E}^{\mathcal{C}_\infty}_{x}\left[ \left(\sum_{s = 0}^{t} \mathds{1}\{X_s = x\} \right)^{2}\right] \le t^{2\eta} + 2t^\eta \mathbb{E}^{\mathcal{C}_\infty}_{x}\left[ \sum_{s = t^\eta}^{t} \mathds{1}\{X_s = x\} \right] + \mathbb{E}^{\mathcal{C}_\infty}_{x}\left[ \left(\sum_{s = t^\eta}^{t} \mathds{1}\{X_s = x\} \right)^{2}\right].
\end{split}
\end{equation*}
As we noted in \eqref{eq:HKUB-compared}, for times $s> t^\eta$ we have that $p^{\mathcal{C}_\infty}_{s}(z, z) \le C/s$ which immediately leads to
\begin{equation}
    \label{eq:Bound-Perc-up-to-t}\mathbb{E}^{\mathcal{C}_\infty}_{x}\left[\sum_{s = t^\eta}^{t} \mathds{1}\{X_s = x\} \right] \le C\log(t).
\end{equation}
Hence, we are left to bound 
\begin{equation*}
    \begin{split}\mathbb{E}^{\mathcal{C}_\infty}_{x}\left[ \left(\sum_{s = t^\eta}^{t} \mathds{1}\{X_s = x\} \right)^{2}\right] &= \mathbb{E}^{\mathcal{C}_\infty}_{x}\left[ \sum_{s = t^\eta}^{t} \mathds{1}\{X_s = x\} \right] \\
    & + 2\mathbb{E}^{\mathcal{C}_\infty}_{x}\left[\sum_{s = t^\eta}^{t} \sum_{r = 0}^{t - s}  \mathds{1}\{X_s = x\} \mathds{1}\{X_{s+r} = x\} \right].
    \end{split}
\end{equation*}
We focus on the second term since the first term coincides with~\eqref{eq:Bound-Perc-up-to-t} and is therefore bounded above by $C\log(t)$. By the simple Markov property we have that
\begin{equation*}
    \mathbb{E}^{\mathcal{C}_\infty}_{x}\left[ \sum_{s = t^\eta}^{t} \sum_{r = 0}^{t - s}  \mathds{1}\{X_s = x\} \mathds{1}\{X_{s+r} = x\} \right] = \sum_{s = t^\eta}^{t} \sum_{r = 0}^{t - s}  p^{\mathcal{C}_\infty}_s(x, x) p^{\mathcal{C}_\infty}_r(x, x).
\end{equation*}
By the same argument as above
\begin{equation*}
    \sum_{s = t^\eta}^{t} \sum_{r = 0}^{t - s}  p^{\mathcal{C}_\infty}_s(x, x) p^{\mathcal{C}_\infty}_r(x, x) \le t^\eta \sum_{s = t^\eta}^{t} \frac{C}{s} + \sum_{s = t^\eta}^{t}\sum_{r = t^\eta}^{t} \frac{C}{s} \frac{C}{r} \le Ct^{2\eta},
\end{equation*}
concluding the proof in the case $\|x\|_\infty = k$. \medskip

We now turn to the proof for $x = 0$, which follows the same lines. In fact, we can apply~\eqref{eq:HKUB-compared} with $x = y= 0$, since $t^\eta \geq k^{\varepsilon\eta} \ge \vartheta_0(\mu)$ (recall that $\vartheta_0(\mu) < \infty$ by~\eqref{eq:Theta-finite} as $0 \in \mathcal{C}_\infty$).
\end{proof}

With these preparations in place, we now turn to the main result of this section. For a fixed realization $\mu \in \{0 \in \mathcal{C}_\infty \}$, we consider the comb graph $G = \mathrm{Comb}(\mathcal{C}_\infty,f_\gamma)$, with $f_\gamma(x) = \varrho_{\ell,\gamma}(\|x\|_\infty )$ (recall~\eqref{eq:Generic-teeth}). The following is the equivalent of Theorem~\ref{theo:2d} in the context of the supercritical percolation cluster. 
\begin{thm}\label{theo:perc}
    Let $p > \frac{1}{2}$, then for a set $\widetilde{\Omega}_0 \subseteq \{0 \in \mathcal{C}_\infty \}$ of full $\mathbb{Q}_0$-probability, the following holds: Let $G = \mathrm{Comb}(\mathcal{C}_\infty,f_\gamma)$ be the comb graph where the underlying graph is $ \mathcal{C}_\infty$ and the tooth profile is given by 
    \begin{equation}
        f_\gamma(x) = \varrho_{\ell,\gamma}(\|x\|_\infty), \qquad \text{with }\varrho_{\ell,\gamma}(k) = \lfloor \log^\gamma(k \vee 1)\rfloor \text{ as in~\eqref{eq:Generic-teeth}},
    \end{equation}
    then 
    \begin{enumerate}
        \item If $\gamma \le 1$ then $G$ has the infinite collision property.
        \item If $\gamma > 1$ then $G$ has the finite collision property.
    \end{enumerate}
\end{thm}
Before we prove the above theorem, we give some informal comments on the strategy: The challenging part is again the proof of the finite collision property for $\gamma > 1$. For the latter, we proceed essentially as in Section~\ref{sec:2D-Comb}, however special care is required since we can only employ bounds for large $k$ (depending on $\mu \in \widetilde{\Omega}_0$). It turns out however that as soon as uniform controls for large enough $k$ (and possibly large enough $t$, depending on $k$) are in force, we can still use an equivalent of Lemma~\ref{lemma:2Expectations2d} to conclude.

\begin{proof}[Proof of Theorem~\ref{theo:perc}]
    Similarly as in Sections~\ref{sec:Fractal-base} and~\ref{sec:2D-Comb}, the infinite collision property (1) in the regime $\gamma \leq 1$ can be proved using the Green kernel criterion~\eqref{eq:Infinite-collision}, as we now explain. For a fixed percolation configuration $\mu \in \{0 \in \mathcal{C}_\infty\}$ and $N \in \mathbb{N}$, consider the (random) set
    \begin{equation}
        D_N(\mu) = \{(x_1,x_2) \in \mathrm{Comb}(\mathcal{C}_\infty,f_\gamma) \, : \, d_{\mathcal{C}_\infty}(0,x_1) \leq N \}.
    \end{equation}
    Recall the definitions of $\Omega^{(j)}$, $1 \leq j \leq 3$ from above~\eqref{eq:BoivinRau-bound}, above~\eqref{eq:Abe-bound}, and the statement of Lemma~\ref{lem:Box-origin-lemma}, respectively. We note that for any $N \in \mathbb{N}$ and $\mu \in  \Omega^{(1)} \cap \Omega^{(2)} \cap \Omega^{(3)} (\subseteq \{0 \in \mathcal{C}_\infty\})$, we have
\begin{equation}
    B_{\mathcal{C}_\infty}(0,N) \subseteq \mathcal{C}_{B_\infty(0,N)}(0) = \mathcal{C}^{N} \subseteq \mathcal{C}^{2N},
\end{equation}
whenever $N \geq N_0''(\mu)$, having used that for $x \in \mathcal{C}_\infty$, $\|x\|_\infty \leq d_{\mathcal{C}_\infty}(0,x)$. We therefore see that on the same event, we have 
\begin{equation}  R_{\mathrm{eff}}^{\mathcal{C}_\infty}(0,B_{\mathcal{C}_\infty}(0,N)^c) \stackrel{\eqref{eq:BoivinRau-bound}}{\geq} c_7 \log N,
\end{equation}
   as well as
   \begin{equation}
R_{\mathrm{eff}}^{\mathcal{C}_\infty}(x,B_{\mathcal{C}_\infty}(0,N)^c) \leq \max_{y \in \mathcal{C}^{2N}} R_{\mathrm{eff}}^{\mathcal{C}_\infty}(x,y) \leq c_8 \log(2N), \text{ for all }x \in B_{\mathcal{C}_\infty}(0,N),
   \end{equation}
   whenever $N \geq N_0(\mu) \vee N_0'(\mu) \vee N_0''(\mu)$, having also used the monotonicity of the effective resistance in the first step. Arguing as in the proof of Theorem~\ref{theo:Weak}, cf.~\eqref{eq:Green-lower-first-instance}, one has (with $0_V = (0,0)$)
   \begin{equation}
       g_{D_N(\mu)}(0_V,0_V) = R_{\mathrm{eff}}^{\mathcal{C}_\infty}(0,B_{\mathcal{C}_\infty}(0,N))^c \geq c \log N,
   \end{equation}
   and since $\gamma \leq 1$, we see in the same way as for~\eqref{eq:Green-upper-first-instance} that
     \begin{equation}
         \sup_{\substack{x = (z,h) \in \mathrm{Comb}(\mathcal{C}_\infty,f_\gamma), \\ \|z\|_\infty \leq N}} R^{G}_{\mathrm{eff}}(x, D_{N}(\mu)^c) \leq c'' \log(N),
     \end{equation}
     provided that $N \geq N_0(\mu) \vee N_0'(\mu) \vee N_0''(\mu)$. The result follows by the Green kernel criterion~\eqref{eq:Infinite-collision}, possibly after adjusting the constants for smaller $N$. \medskip

     We now turn to (2). We will rely heavily on the bounds for $\mathbb{Z}^2$ derived in the previous Section, and use the regularity of the cluster on large scales, see~\eqref{eq:Theta-small} and Lemma~\ref{lemma:MomentsOccTimePerc}. We also need a variant of Lemma~\ref{lemma:HorHittingTime} adjusted to the percolation cluster, which is stated next.
\begin{lem}
\label{lem:Exit-time-percolation}
    Let $p > \frac{1}{2}$. There exists a set $\Omega^{(6)} \subseteq \{0 \in \mathcal{C}_\infty
    \}$ of full $\mathbb{Q}_0$-probability such that the following holds: For every $\mu \in \Omega^{(6)}$, there exists $K_2(\mu) < \infty$ with the property
    \begin{equation}
    \label{eq:Good-exit-controls-perc}
        \begin{minipage}{0.8\textwidth} $\mathbb{P}_x(L_{x,k} < k^2/\lambda) \leq C\exp\left(-c\lambda \right)$ for every $1\leq \lambda \leq (k\log^\gamma(k))^{\frac{1}{3}}$, $x \in B_\infty(0,e^{k^c}) \cap \mathcal{C}_\infty$, whenever $k \geq K_2(\mu)$.
        \end{minipage}
    \end{equation}
\end{lem}
\begin{proof}
    This follows from a careful application of~\cite[Proposition 3.7]{barlow2004random} in our context. We provide a proof in the Appendix~\ref{appendix}.
\end{proof}
With these preparations, we now aim at adapting the results of the previous section to the percolation cluster. We start with a generalization of the crucial Proposition~\ref{prop:HeatKernelZ2}. For $\mu \in \{0 \in \mathcal{C}_\infty\}$, we abbreviate by $q_t^{\mu}(x,y) = p_t^{\mathrm{Comb}(\mathcal{C}_\infty,f_\gamma)}(x,y)$ the heat kernel on the comb graph, where $t \geq 0$, $x, y \in V$, and use a similar notation for the Green kernel.

\begin{prop} 
\label{prop:Main-HK-bound-cluster}
    Fix $\varepsilon\in (0,1)$. For $\mu \in \Omega^{(4)} \cap \Omega^{(5)} \cap \Omega^{(6)}$, there exist $K_3(\mu,\varepsilon) < \infty$ such that the following holds for every $k \geq K_3(\mu,\varepsilon) \vee \vartheta_0^{40/\varepsilon}(\mu)$: For any
    $x = (z,h) \in V$ with  $\|z\|_\infty = k$ and for every $t \geq k^\varepsilon$, one has for some $c(\varepsilon,\mu) > 0$ that
    \begin{equation}
    \label{eq:HK-perc-Main-bound}
        q^\mu_t(0, x) \le 
        \begin{cases}
            \frac{c(\varepsilon,\mu)}{t} & t \ge k^2\log^{\gamma}(k)\\
            \frac{c(\varepsilon,\mu)}{k^2\log^\gamma(k)} & t < k^2\log^{\gamma}(k).
        \end{cases}
    \end{equation}
\end{prop}
\begin{proof}
The proof is carried out in multiple steps. First,
we show the following equivalent of Lemma~\ref{lemma:HeatKernelZ2k}: Let $x = (z,0)$ with $\|z\|_\infty = k$ and $z \in \mathcal{C}_\infty$ or $x = 0_V = (0,0)$. For any $\varepsilon \in (0,1)$ and $t \geq k^\varepsilon$ we have that for $c > 0$, 
    \begin{equation}
    \label{eq:Point-to-point-perc-bound}
        q_t^\mu(x,x) \leq \frac{c}{t},
    \end{equation}
    provided $k \geq K_1(\mu,\frac{1}{40},\varepsilon) \vee K_2(\mu) \vee \vartheta_0^{40/\varepsilon}(\mu)$. We begin by addressing the case $\|z\|_\infty = k$. To establish~\eqref{eq:Point-to-point-perc-bound} in this case, we first show, similarly as in the proof of Lemma~\ref{lemma:HeatKernelZ2k}, that $g^\mu_{[t/2,t]}(x,x) \leq c$ under these conditions. 
     Define the events $A$ and $B$ as in~\eqref{eqn:LargeDeviationEvents}.
    Note that by~\eqref{eq:Theta-small} we see that $t_1 = c_2 \log^{-\gamma}(t)t/2$ and $t_2 = c_3 \log^{-\gamma}(t)t$ both exceed $\vartheta_x$ for all $\|z\|_\infty = k$ if $t \geq k^\varepsilon$ and $k \geq K_0(\mu) \vee C(\varepsilon)$. Therefore, we have (using the same notation for random variables $B^i = \{B^i_z \}_{i \geq i}$ as below~\eqref{eqn:Bzed})
    \begin{equation}         \begin{split}      \mathbb{E}_{x}\left[ \sum_{s = t/2}^{t} \mathds{1}\{X_s = x\} \mathds{1}\{A \cap B\} \right] &\le \mathbb{E}_{x}\left[ \sum_{i = c_2 \log^{-\gamma}(t)t/2}^{c_3 \log^{-\gamma}(t)t} B^i_z \mathds{1}\{\tilde{X}_i = z\}\right]\\         &\le \mathbb{E}_{x}\left[\sum_{i = c_2 \log^{-\gamma}(t)t/2}^{c_3 \log^{-\gamma}(t)t} \mathds{1}\{\tilde{X}_i = z\}\right] \mathbb{E}_{x}\left[ B^1_z \right]\\         & = \sum_{i = c_2 \log^{-\gamma}(t)t/2}^{c_3 \log^{-\gamma}(t)t} p^{\mathcal{C}_\infty}_i(z,z) \mathbb{E}_{x}\left[ B^1_z \right] \\         
    & \stackrel{\eqref{eq:HKUB-compared}}{\le} Cg^{\mathbb{Z}^2}_{[t_1,t_2]}(x,x)\mathbb{E}_{x}\left[ B^1_z \right].
    \end{split} \end{equation}
    On the other hand, we can write using the Markov property
    \begin{equation*}
        \begin{split}
            \mathbb{E}_{x}\left[ \sum_{s = t/2}^{t} \mathds{1}\{X_s = x\} \mathds{1}\{A^c\} \right] \le \mathbb{P}_x\left(A^c\right) \sup_{y \in B_{\mathcal{C}_\infty}(x, t)} g^{\mathcal{C}_\infty}_{[0, t/2]}(y, z) \le \mathbb{P}_x\left(A^c\right) g^{\mathcal{C}_\infty}_{[0, t/2]}(z, z).
        \end{split}
    \end{equation*}
    Note that by Lemma~\ref{lemma:MomentsOccTimePerc} (and the fact that the random variable $\sum_{i = 0}^{t/2} \mathbbm{1}\{X_i = z \}$ is integer valued) we have $g^{\mathcal{C}_\infty}_{[0, t/2]}(z, z) \le Ct^{1/10}$, whenever $k \geq K_1(\mu,1/40,\varepsilon
    )$. Moreover, by Proposition~\ref{prop:BadEventsCluster} below we conclude that 
    \begin{equation*}
        \mathbb{E}_{x}\left[ \sum_{s = t/2}^{t} \mathds{1}\{X_s = x\} \mathds{1}\{A^c\} \right] \le C t^{1/20}t^{-1/4}.
    \end{equation*}
    In a similar fashion, using the Cauchy-Schwarz inequality and Proposition~\ref{prop:BadEventsCluster} below and Lemma~\ref{lemma:MomentsOccTimePerc}
    \begin{equation*}
        \mathbb{E}_{x}\left[ \sum_{s = t/2}^{t} \mathds{1}\{X_s = x\} \mathds{1}\{B^c\} \right] \le C t^{1/20} t^{-1/10},
    \end{equation*}
    provided that $k \geq K_1(\mu,\frac{1}{40},\varepsilon) \vee K_2(\mu)$. The proof for the case $x =(0,0)$ follows from the same argument, one only needs to observe that $t_1, t_2 \ge \vartheta_0(\mu)$ since $t_1, t_2 \ge t^{1/40}$ and $t^{1/40} \ge \vartheta_0(\mu)$.
We are left with the proof of the following proposition, which is an analogue of Proposition~\ref{prop:ChernovBound}.

\begin{prop}\label{prop:BadEventsCluster}
    Consider the events defined in \eqref{eqn:LargeDeviationEvents}. Then, for every $\mu \in \Omega^{(4)} \cap \Omega^{(5)}$, for every $x = (z,0)$ with $z = 0$ or $z \in \mathcal{C}_\infty$ such that $\|z\|_\infty = k \geq K_1(\mu,\frac{1}{40},\varepsilon) \vee 4 \vee \vartheta_0^{40/\varepsilon}(\mu)$ and $t \ge k^\varepsilon$, there exist constants $C, C' > 0$ such that 
    \begin{equation*}
        \mathbb{P}_x\left( A^c \right)\le \frac{C}{t^{1/2}},
    \end{equation*}
    and
    \begin{equation*}
        \mathbb{P}_x\left( B^c \right)\le \frac{C'}{t^{1/5}}.
    \end{equation*}
\end{prop}

\begin{proof}
    The argument is similar to its counterpart given for Proposition~\ref{prop:ChernovBound}. We will only describe the necessary modifications. We will retain the notation used in the proof of Proposition~\ref{prop:ChernovBound}. \medskip

    \noindent \textbf{The event $A^c$:} As we briefly explain, the proof given for the corresponding event in Proposition~\ref{prop:ChernovBound} carries over essentially without changes. Indeed, we need to bound the probability on the right-hand side of~\eqref{eqn:InversionEventA}, which is done conditionally on all possible trajectories $\omega$ of $\tilde{X}$ (which is now a walk on $\mathcal{C}_\infty \subseteq \mathbb{Z}^2$). Since every trajectory $\omega$ can be viewed as a trajectory in $\mathbb{Z}^2$, we can conclude as in~\eqref{eq:Bound-A-event} (which shows that the exponent can be chosen to be $\frac{1}{2}$). \medskip

    \noindent \textbf{The event $B^c$:} We again follow the corresponding proof in the case of $\mathbb{Z}^2$, and derive a bound on the probability on the right-hand side of~\eqref{eqn:InversionEventB}. To that end, we use that for every $k \geq K_1(\mu,\frac{1}{40}, \varepsilon) \vee \vartheta_0^{40/\varepsilon}(\mu)$, we see that by the result of Lemma~\ref{lemma:MomentsOccTimePerc} (with $\eta = \frac{1}{40}$), one has
    \begin{align*}
    \mathbb{P}_x\left( \ell_{B(0, t^{1/3})}(0, c_2 \log^{-\gamma}(t) t) > \frac{1}{2} c_2 \log^{-\gamma}(t) t\right) &\le C \frac{\mathbb{E}_x\left[ \ell_{B_{\mathcal{C}_\infty}(0, t^{1/3})}(0, c_2 \log^{-\gamma}(t) t)\right]}{t\log^{-\gamma}(t)}\\
    & \le \frac{\sum_{x \in B_{\mathcal{C}_\infty}(0, t^{1/3})} \mathbb{E}^{\mathcal{C}_\infty}_z\left[ \sum_{s = 0}^{t} \mathds{1}\{X_s = z\}\right]}{t\log^{-\gamma}(t)}\\
    & \le C \frac{t^{2/3 + 1/20}}{t\log^{-\gamma}(t)} \le C t^{-1/5},
\end{align*}
using also $Z \leq Z^2$ for the integer-valued random variable $Z = \sum_{s = 0}^{t} \mathds{1}\{X_s = z\}$. This is the only step that requires a modification when dealing with the percolation cluster compared to $\mathbb{Z}^2$, and since  the remainder of the proof can be followed line by line, we omit repeating it.
\end{proof}
Having established~\eqref{eq:Point-to-point-perc-bound}, we now show how the proof of~\eqref{eq:HK-perc-Main-bound} is concluded. We first note that by the same argument as in Lemma~\ref{lemma:HeatZ2CS}, we see that for all $t > 0$ and $x = (z,0)$ with $\|z\|_\infty = k$, $z \in \mathcal{C}_\infty$ and assuming that $k \geq K_1(\mu,\frac{1}{40},\varepsilon) \vee  \vartheta_0^{40/\varepsilon}(\mu)$, the Cauchy-Schwarz inequality gives us 
\begin{equation}
    q^\mu_t(0,x) \leq \frac{c}{t}.
\end{equation}
 It remains to prove the equivalent of Lemma~\ref{lemma:SmallTimeZ2}, namely that for $x = (z,h)$ with $\|z\|_\infty = k$, $h \geq 0$, and $t \leq k^2\log^\gamma(k)$, one has
\begin{equation}
\label{eq:Small-times-perc}
    q_t^\mu(0,x)  \leq \frac{c}{k^2\log^\gamma(k)}.
\end{equation}
First note again that the bound is trivial for $k < t$, so we assume $k \leq t \leq k^2 \log^\gamma(k)$. With this in mind, we write similarly as in Lemma~\ref{lemma:ExitCombLDP} and~\ref{lemma:HorExitCombZ2}:
 \begin{equation*}
       \mathbb{P}_{x}\left(T_{x, k} < t \right) = \mathbb{P}_{x}\left(T_{x, k} < t, L_{x, k}< k^2/\lambda \right) + \mathbb{P}_{x}\left(T_{x, k} < t, L_{x, k} \ge k^2/\lambda \right),
   \end{equation*}
for $\lambda > 0$. The first summand can be treated exactly as in Lemma~\ref{lemma:HorExitCombZ2}, noting that in particular, one has $\lambda = \left( \frac{k^2\log^\gamma(k)}{t} \right)^{\frac{1}{3}}$. Inserting the values $k \leq t \leq k^2\log^\gamma(k)$, we readily see that $1 \leq \lambda \leq (k\log^\gamma(k))^{\frac{1}{3}}$. Therefore we can apply Lemma~\ref{lem:Exit-time-percolation} provided that $k \geq K_2(\mu)$ (note also that if $\|z\|_\infty > e^{k^c}$, the bound is trivially true since $t \leq k^2\log^\gamma(k)$ and $q^\mu_t(0,x) = 0$ in this case). This proves~\eqref{eq:Small-times-perc} and concludes the proof of~\eqref{eq:HK-perc-Main-bound}.
\end{proof}
We also note that using same argument as in Lemma~\ref{lemma:HKGaussZ2} (see also Lemma~\ref{lemma:HKonFractComb}), we then obtain that for $x = (z,h)$ with $\|z\|_\infty= k$, $z \in \mathcal{C}_\infty$, $h \geq 0$, one has
\begin{equation}
\label{eq:Gaussian-bounds-on-teeth}
    q_t^\mu(0,x) \leq \frac{c}{t} e^{- \frac{h^2}{c't}},
\end{equation}
for any $t \ge 0$ provided that $k \ge K_1(\mu,\varepsilon) \vee K_2(\mu) \vee \vartheta_0^{40/\varepsilon}(\mu)$ (indeed, note that the expression is zero if $t \leq k$, and~\eqref{eq:Point-to-point-perc-bound} holds as long as $t \geq k^\varepsilon$).
We can now finish the proof of Theorem~\ref{theo:perc} rather straightforwardly.  Suppose that $\mu \in \Omega^{(4)} \cap \Omega^{(5)} \cap \Omega^{(6)}$ and corresponding $K_3(\mu) < \infty$, $\vartheta_0^{40/\varepsilon}(\mu) < \infty$ are given as in the statement of Proposition~\ref{prop:Main-HK-bound-cluster}. We again define random variables $\mathcal{Z}^\mu_{k, \ell}$ as the number of collisions in $Q_{k, \ell}^\mu \coloneqq \{(u, h) \in G \colon \|u\|_\infty = k, 0 \le h \le \ell\}$ and $\tilde{\mathcal{Z}}_{k, \ell}^\mu$ as the number of collisions in $\widetilde{Q}^\mu_{k, \ell} \coloneqq \{(u, \ell) \in G \colon \|u\|_\infty = k, \ell/3 \le h \le 2\ell/3\}$ (we stress here that $G$ depends on the realization $\mu$). We claim that analogously to Lemma~\ref{lemma:2Expectations2d}, we have for any $k$ large enough (possibly depending on $\mu$) the bounds
\begin{equation}
    \begin{split}
\mathbb{E}_{0}[\mathcal{Z}^\mu_{k, \ell}] & \le \ell k^{-1}\log^{-\gamma}(k) , \\
\mathbb{E}_{0}[\mathcal{Z}^\mu_{k, \ell} | \, \tilde{\mathcal{Z}}^\mu_{k, \ell}>0 ] & \ge c\ell.
    \end{split}
\end{equation}
Indeed, we take $k$ large enough such that $k \geq \vartheta_0^{40/\varepsilon}(\mu) \vee K_3(\mu)$, then 
\begin{align*}
        \mathbb{E}_{0}[\mathcal{Z}^\mu_{k, \ell}] &= \sum_{t = 0}^\infty \sum_{x \in Q^\mu_{k, \ell}} q^\mu_{t}(0, x)^2 \\
        &\le \ell \sum_{k \leq t < k^2 \log^\gamma(k)} \frac{ck}{k^{4} \log^{2\gamma}(k)} + \ell\sum_{t \ge k^2 \log^\gamma(k)} \frac{ck}{t^{2}}\\
        &\le c\ell k^{-1}\log^{-\gamma}(k),
    \end{align*}
having used~\eqref{eq:Gaussian-bounds-on-teeth} and~\eqref{eq:HK-perc-Main-bound}. The remainder of the proof is exactly as in the previous section, and is therefore omitted. Finally, we see that the statement of Theorem~\ref{theo:perc} is valid for $\widetilde{\Omega}_0 \coloneqq \bigcap_{j =1}^6 \Omega^{(j)}$, and this finishes the proof. \end{proof}
\appendix

\section{Auxiliary results and some proofs}
\label{appendix}

\noindent In this appendix, we include two technical proofs of that draw heavily on existing literature. \medskip

We start with introducing some useful notation. We define, for $\beta \geq 2$ and $k,t \geq 1$, the quantity
\begin{equation}\label{eqn:ExponentRWs}
    \Phi_\beta(t, k) \coloneqq \left( \frac{k^\beta}{t} \right)^{\frac{1}{\beta - 1}},
\end{equation}
We remark that in the case of $\mathbb{Z}^d$ one can insert $\beta = 2$ and obtain the simple expression $\Phi_2(t, k) = \frac{k^2}{t}$. We now prove Lemma~\ref{lemma:HorHittingTime}, a well-known result on the exit times of random walks on fractal graphs. The proof follows along the same lines of \cite[Lemma~3.7]{BarlowCoulhonKumagai}.

\begin{proof}[Proof of Lemma~\ref{lemma:HorHittingTime}]
    Recall that $x = (x_1,x_2)$.
    Note that the law of $L_{x,k}$ under $\mathbb{P}_x$ is the same as the law of $T_{B_{\tilde{G}}(x_1,k)}$ under $\mathbb{P}^{\tilde{G}}_{x_1}$. Without loss of generality let $x_1 = o$ abbreviate $B_{\tilde{G}}(z,r)$ by $B(z,r)$ for $z \in \tilde{V}$ and $r \geq 0$. We begin by observing that for $n, r > 0$ there exist $p\in (0, 1)$ and $A > 0$ such that
    \begin{equation*}
        \mathbb{P}_{o}^{\tilde{G}}(T_{ B(o, r)} \le n) \le p + n \left(\sup_{x \in B(o, r)} \mathbb{E}^{\tilde{G}}_x[T_{ B(o, 2r)}]\right)^{-1}.
    \end{equation*}
    \begin{equation*}
        \mathbb{E}^{\tilde{G}}_o[T_{ B(o, r)} ] \le 2n + \mathbb{P}_o^{\tilde{G}}(T_{ B(o, r)} > n) \sup_{x \in B(0, r)} \mathbb{E}_x^{\tilde{G}}[T_{ B(o, 2r)}].
    \end{equation*}
    Thus we obtain 
    \begin{equation*}
        \mathbb{E}^{\tilde{G}}_o[T_{ B(o, r)} ] \le 2n + \left( 1 - \mathbb{P}_o^{\tilde{G}}(T_{ B(o, r)} > n) \right)\sup_{x \in B(o, r)} \mathbb{E}_x^{\tilde{G}}[T_{ B(o, 2r)}],
    \end{equation*}
    rearranging and observing that
    \begin{equation*}
        p = \mathbb{E}_o^{\tilde{G}}[T_{ B(o, r)} ] \left( \sup_{x \in B(o, r)} \mathbb{E}^{\tilde{G}}_x[T_{ B(o, 2r)}] \right)^{-1} \in (0, 1).
    \end{equation*}
We also note that by the assumption given in \eqref{eqn:ExitCondition}, we have that $\mathbb{E}_0[T_{ B(0, r)} ] = A r^\beta$ for some constant $A>0$.

Consider the first time the random walk exits the ball of radius $\frac{k}{\lambda}$ and denote it $T_{\partial B(0, \frac{k}{\lambda})}$. Define recursively the stopping times
    \begin{align*}
        \tau_1 &\coloneqq T_{ B(o, \frac{k}{\lambda})}\\
        \tau_{\ell+1} &\coloneqq T_{ B(X_{\tau_\ell}, \frac{k}{\lambda})}.
    \end{align*}
    Clearly, we have that
    \begin{equation*}
        \sum_{i = 1}^\lambda \tau_i \le T_{ B(o, k)} \stackrel{(\mathrm{d})}{=} L_{x, k}.
    \end{equation*}
    We also observe that if $\mathcal{G}_i$ denotes the $\sigma$-field generated by the walk up to time $\sum_{j = 1}^i \tau_i$ then
    \begin{equation*}
        \mathbb{P}_o^{\tilde{G}}(\tau_{i+1} \le t \mid \mathcal{G}_i) \le p + A'\frac{t}{(k/\lambda)^\beta}
    \end{equation*}
    We are now in position to apply \cite[Lemma~A.8]{Barbook} with the parameters
    \begin{equation*}
        a = p, \quad \text{and} \quad b = A'\frac{\lambda^\beta}{k^\beta}.
    \end{equation*}
    We obtain
    \begin{equation*}
        \mathbb{P}_o^{\tilde{G}}\left(\sum_{i = 1}^\lambda \tau_i  \le t\right) \le \exp\left( \left( \frac{A' \lambda t}{p (k/\lambda)^\beta}\right)^{1/2} - \lambda \log\left(\frac{1}{p}\right) \right) \le e^{C (tk^{-\beta} \lambda^{\beta+1})^{1/2} - c \lambda}.
    \end{equation*}
    We now need to optimise over the choice of $\lambda$. First assume that $k^\beta t^{-1} > a$ for some fixed $a>0$. Select $\lambda_0 \in \mathbb{N}$ to be the largest $\lambda$ such that
    \begin{equation*}
        \frac{c}{2} \lambda > C (tk^{-\beta} \lambda^{\beta+1})^{1/2}.
    \end{equation*}
    Observe that for $a$ large enough this holds for $\lambda$ small enough. One can notice that
    \begin{equation*}
        \lambda_0^{\beta - 1} \le t^{-1}k^{\beta} (c^2/4C^2) \le (\lambda_0 + 1)^{\beta - 1},
    \end{equation*}
    which implies the result in the regime $k^{\beta}t^{-1} > a$. On the other hand when $k^{\beta}t^{-1} \le a$ one can adjust the constant in the statement to make the result hold true. The proof is finished.
\end{proof}

We turn to the proof of Lemma~\ref{lem:Exit-time-percolation}. 

\begin{proof}[Proof of Lemma~\ref{lem:Exit-time-percolation}]
    By~\cite[Proposition 3.7]{barlow2004random}, we have the following: If $x \in \mathcal{C}_\infty$ and $B_{\mathcal{C}_\infty}(x,R)$ is ``very good'', then for $t > 0$, $\rho > 0$ with
\begin{equation} 
\label{eq:Conditions-Barlow-VeryGoodBall}
    \rho \leq R, \quad cN_B^2(\log N_B)^{\frac{1}{2}}\rho \leq t, \quad \text{and }t \leq T_B' \end{equation}
    with 
    \begin{equation}
        N_B(x) \leq R^{\frac{1}{4}}, \qquad T_B' = c'\frac{R^2}{\log R},
    \end{equation}
    then, 
    \begin{equation}
        \label{eq:Barlow-very-good-ball-conclusion}\mathbb{P}_x^{\mathcal{C}_\infty}(T_{B_{\mathcal{C}_\infty}(x,\rho)}< t) \leq c_2\exp\left( -\frac{\rho^2}{t}c_3\right)
    \end{equation}
(the precise definition of very good balls can be found in~\cite[Definition 1.7]{barlow2004random} but is not central to our argument to proceed). Moreover, one knows by~\cite[Section 2]{barlow2004random} that there are random variables $(\mathcal{T}_x)_{x \in \mathbb{Z}^d}$ such that for all $\mu \in \Omega'\cap \{0 \in \mathcal{C}_\infty\}$ and $x \in \mathcal{C}_\infty$, the following holds
\begin{equation}
\label{eq:Growth-very-good-range}
    \begin{cases}
        \mathcal{T}_x(\mu) < \infty, \\
        \text{for all $R \geq \mathcal{T}_x$, $B_{\mathcal{C}_\infty}(x,R)$ is very good with $N_{B(x)} \leq R^{\frac{1}{4}}$,} \\ 
        \text{for all $z \in \mathbb{Z}^d$, $r \geq 1$, one has $
        \mathbb{Q}(\mathcal{T}_z \geq r) \leq C\exp(-c r^\vartheta
        ))
        $,}
    \end{cases}
\end{equation}
see also~\cite[Theorem 1.13]{sapozhnikov2017long-range} and~\cite[Remark 1.21]{sapozhnikov2017long-range}. With~\eqref{eq:Growth-very-good-range}, we now argue similarly as in~\eqref{eq:Bad-events-Heat-kernel} and define the events
\begin{equation}
    \mathrm{Bad}_R = \bigcup_{x \in B_\infty(0,\exp(R^{\frac{\vartheta}{3}}))} \{\mathcal{T}_x \geq R \}.
\end{equation}
We see that
\begin{equation}
    \sum_{R =1}^\infty \mathbb{Q}(\mathrm{Bad}_R) \leq \sum_{R = 1}^\infty C(2\exp(R^{\frac{\vartheta}{3}})+1)^2 \exp\left(-cR^\vartheta \right) < \infty.
\end{equation}
Therefore by the Borel-Cantelli lemma there a set $\Omega'$ of full $\mathbb{Q}$-probability and for every $\mu \in \Omega'$, there is $R_0(\mu) < \infty$ such that for all $x \in B_\infty(0,\exp\{R^{\frac{\vartheta}{3}}\}) \cap \mathcal{C}_\infty$, $B_{\mathcal{C}_\infty}(x,R)$ is very good with $N_{B(x)} \leq R^{\frac{1}{4}}$, for every $R \geq R_{0}(\mu)$. We will now argue that~\eqref{eq:Good-exit-controls-perc} follows from the latter observation. To that end, let $\epsilon > 0$ to be chosen later. We let $K_0(\mu) = R_0(\mu)^\frac{1}{1+\epsilon}$ and set
   $ \rho = k$ and $R = k^{1+\epsilon} (\geq k)$ in~\eqref{eq:Conditions-Barlow-VeryGoodBall}. With these choices, we see that the aforementioned conditions are fulfilled provided that
   \begin{equation}
       ck^{\frac{1+\epsilon}{2}} \log\left( k^{1+\epsilon} \right)^{\frac{1}{2}} k \leq \left(\frac{k^5}{\log^\gamma(k)}\right)^{\frac{1}{3}} \leq t \leq k^2 \leq \frac{c'k^{2+2\epsilon}}{(1+\epsilon)\log(k)},
   \end{equation}
   for every $k \geq K_0(\mu) \vee C(\epsilon,\gamma)$, where we choose $\epsilon > 0$ small enough. Setting $\lambda = \frac{k^2}{t}$, we observe that with the above~\eqref{eq:Barlow-very-good-ball-conclusion} yields
   \begin{equation}
       \mathbb{P}^{\mathcal{C}_\infty}_x\left(T_{B_{\mathcal{C}_\infty}}(x,k) < \frac{k^2}{\lambda} \right) \leq c_2 \exp\left(-c_3 \lambda \right)
   \end{equation}
for every $x \in B_\infty(0,R^{\frac{\vartheta(1+\epsilon)}{3}}) \cap \mathcal{C}_\infty$, $k \geq K_0(\mu) \vee C(\epsilon,\gamma)$, and $1 \leq \lambda \leq (k^2\log^\gamma(k))^{\frac{1}{3}}$. Moreover, since $B_{\mathcal{C}_\infty}(x,k) \subseteq B_\infty(x,k)$, we can replace $T_{B_{\mathcal{C}_\infty}}$ in the display by $T_{B_\infty(x,k)}$ and obtain the same bound. Since $L_{x,k}$ under $\mathbb{P}_x$ has the same law as $T_{B_\infty(x,k)}$ under $\mathbb{P}_x^{\mathcal{C}_\infty}$, this proves the claim.
\end{proof}

\noindent \textbf{Acknowledgments.} MN was partially supported by Hong Kong RGC grants ECS 26301824 and GRF 16303825. The authors are grateful to Noam Berger and David Croydon for helpful discussions.

\bibliographystyle{plain}
\bibliography{biblio}

\begin{thebibliography}{10}

\bibitem{abe2015effective}
Y.~Abe.
\newblock Effective resistances for supercritical percolation clusters in
  boxes.
\newblock {\em Ann. Inst. Henri Poincar\'e{} Probab. Stat.}, 51(3):935--946,
  2015.

\bibitem{barlow2004random}
M.~Barlow.
\newblock Random walks on supercritical percolation clusters.
\newblock {\em Ann. Probab.}, 32(4):3024--3084, 2004.

\bibitem{Barbook}
M.~Barlow.
\newblock {\em Random walks and heat kernels on graphs}, volume 438 of {\em
  London Mathematical Society Lecture Note Series}.
\newblock Cambridge University Press, Cambridge, 2017.

\bibitem{BarlowCoulhonKumagai}
M.~Barlow, T.~Coulhon, and T.~Kumagai.
\newblock Characterization of sub-{G}aussian heat kernel estimates on strongly
  recurrent graphs.
\newblock {\em Comm. Pure Appl. Math.}, 58(12):1642--1677, 2005.

\bibitem{BPS}
M.~Barlow, Y.~Peres, and P.~Sousi.
\newblock Collisions of random walks.
\newblock {\em Ann. Inst. Henri Poincar\'e{} Probab. Stat.}, 48(4):922--946,
  2012.

\bibitem{boivin2013existence}
D.~Boivin and C.~Rau.
\newblock Existence of the harmonic measure for random walks on graphs and in
  random environments.
\newblock {\em J. Stat. Phys.}, 150(2):235--263, 2013.

\bibitem{CWZ08}
D.~Chen, B.~Wei, and F.~Zhang.
\newblock A note on the finite collision property of random walks.
\newblock {\em Statist. Probab. Lett.}, 78(13):1742--1747, 2008.

\bibitem{ChenChen10}
X.~Chen and D.~Chen.
\newblock Two random walks on the open cluster of {$\Bbb Z^2$} meet infinitely
  often.
\newblock {\em Sci. China Math.}, 53(8):1971--1978, 2010.

\bibitem{Chenchen}
X.~Chen and D.~Chen.
\newblock Some sufficient conditions for infinite collisions of simple random
  walks on a wedge comb.
\newblock {\em Electron. J. Probab.}, 16:no. 49, 1341--1355, 2011.

\bibitem{CroydonDeAmbroggio}
David~A. Croydon and Umberto~De Ambroggio.
\newblock Triple collisions on a comb graph.
\newblock {\em To appear in Electron. J. Probab.}, 2024+.

\bibitem{dembo2001thick}
A.~Dembo, Y.~Peres, J.~Rosen, and A.~Zeitouni.
\newblock Thick points for planar brownian motion and the
  {E}rd{\H{o}}s-{T}aylor conjecture on random walk.
\newblock {\em Acta Math.}, 186(239):270, 2001.

\bibitem{devulder2025infinitely}
A.~Devulder.
\newblock Infinitely many collisions between a recurrent simple random walk and
  arbitrary many transient random walks in a subballistic random environment.
\newblock {\em Electron. Comm. Probab.}, 30:1--11, 2025.

\bibitem{DGP1}
A.~Devulder, N.~Gantert, and F.~P\`ene.
\newblock Collisions of several walkers in recurrent random environments.
\newblock {\em Electron. J. Probab.}, 23:Paper No. 90, 34, 2018.

\bibitem{DGP2}
A.~Devulder, N.~Gantert, and F.~P\`ene.
\newblock Arbitrary many walkers meet infinitely often in a subballistic random
  environment.
\newblock {\em Electron. J. Probab.}, 24:Paper No. 100, 25, 2019.

\bibitem{erdos1960some}
P.~Erd\H{o}s and S.~J. Taylor.
\newblock Some problems concerning the structure of random walk paths.
\newblock {\em Acta Math. Acad. Sci. Hungar}, 11:137--162, 1960.

\bibitem{grimmett1999percolation}
G.~Grimmett.
\newblock {\em Percolation}.
\newblock Springer, 1999.

\bibitem{HH}
N.~Halberstam and T.~Hutchcroft.
\newblock Collisions of random walks in dynamic random environments.
\newblock {\em Electron. J. Probab.}, 27:Paper No. 8, 18, 2022.

\bibitem{HP}
T.~Hutchcroft and Y.~Peres.
\newblock Collisions of random walks in reversible random graphs.
\newblock {\em Electron. Commun. Probab.}, 20:no. 63, 6, 2015.

\bibitem{KP}
M.~Krishnapur and Y.~Peres.
\newblock Recurrent graphs where two independent random walks collide finitely
  often.
\newblock {\em Electron. Comm. Probab.}, 9:72--81, 2004.

\bibitem{LPBook}
D.~A. Levin and Y.~Peres.
\newblock {\em Markov chains and mixing times}.
\newblock American Mathematical Society, Providence, RI, 2017.
\newblock Second edition, With contributions by Elizabeth L. Wilmer, With a
  chapter on ``Coupling from the past'' by James G. Propp and David B. Wilson.

\bibitem{LP}
R.~Lyons and Y.~Peres.
\newblock {\em Probability on trees and networks}, volume~42 of {\em Cambridge
  Series in Statistical and Probabilistic Mathematics}.
\newblock Cambridge University Press, New York, 2016.

\bibitem{sapozhnikov2017long-range}
A.~Sapozhnikov.
\newblock Random walks on infinite percolation clusters in models with
  long-range correlations.
\newblock {\em Ann. Probab.}, 45(3):1842--1898, 2017.

\bibitem{Wat23}
S.~Watanabe.
\newblock Infinite collision property for the three-dimensional uniform
  spanning tree.
\newblock {\em Int. J. Math. Ind.}, 15:Paper No. 2350005, 10, 2023.

\end{thebibliography}

\end{document}